\documentclass[11pt]{article}
\usepackage[margin=1.05in]{geometry}
\usepackage{amsmath,amssymb,amsthm}
\usepackage{array}

\newtheorem{theorem}{Theorem}[section]
\newtheorem{proposition}[theorem]{Proposition}
\newtheorem{lemma}[theorem]{Lemma}
\newtheorem{conjecture}[theorem]{Conjecture}
\newtheorem{problem}[theorem]{Problem}
\theoremstyle{remark}
\newtheorem{remark}[theorem]{Remark}

\newcommand{\Q}{\mathbb{Q}}
\newcommand{\Z}{\mathbb{Z}}
\newcommand{\R}{\mathbb{R}}
\newcommand{\D}{\mathcal{D}}
\newcommand{\Pp}{\mathcal{P}}
\allowdisplaybreaks

\title{Forced Shadows of an Obstructed\\ Hyperbolic Kac--Moody
Denominator\thanks{DOI: \texttt{10.5281/zenodo.21973291}}}
\author{Eungang Cho}
\date{August 2026}

\begin{document}
\maketitle

\begin{abstract}
The four orders of the quaternion algebra $B_6$ carry four reflective wall
data on lattices of signature $(3,2)$; three integrate to Borcherds
denominators and one is obstructed. The failed denominator survives as a
weakly harmonic Maass form, and we prove its shadow is a Hecke eigenform on
the line of the newform 6.4.a.a, with zero twist component. The mechanism is
invariance selection: the obstruction functional is invariant under the
discriminant isometry group, whose invariants in $S_{5/2}$ are
one-dimensional; the same mechanism, verified at quaternion discriminants
$10$ and $22$, places the shadows there on 10.4.a.a and 22.4.a.c. On the weight-$1/2$ layer we prove a determination theorem:
the canonical form exists and is unique precisely when the obstruction space
vanishes, and among the 71 discriminants below $230$ this happens exactly for
$D \in \{6, 10, 22\}$, the genus-zero compact Shimura curves, whose
maximal-order ternaries are reflective with integral Weyl chambers of ranks
$3, 4, 4$; completeness beyond that range is reduced to an estimate on a
quadratic Dedekind-type sum, given a bound on the Gauss-sum term. Parity confines this
layer to odd channels; on the obstructed orientation the section layer is
obstructed outside an explicit 40-element locus of orientations (a double
shadow: CM, 36.2.a.a, at weight $3/2$ and newform at weight $5/2$), while
the deck-symmetric directions instead carry a unique canonical weight-$1/2$
form. The defect invariant satisfies $\|\Xi\|^2 \cdot \langle v_+, v_+\rangle
= 144$ exactly and equals $L(f,2)/48\pi^2\langle f,f\rangle$ to 31 digits. On the
section layer the Petersson geometry is rigid: the Gram matrix of
$S_{3/2}(\rho_4)$ is a single transcendental multiple of an exact rational
form, and that transcendental is identified, to 40 digits, as
$3\,\Gamma(1/3)^3/2^{7/3}\pi^2$: the weight-$3/2$ shadow norms lie in the
Chowla--Selberg ring; the weight-$5/2$ norm is numerically excluded from it.
\end{abstract}

\section{Background: from a hyperbolic matrix to a harmonic Maass form}

This section defines every object in the title and fixes notation; a reader
familiar with Borcherds products may skim to the notation table and proceed
to \S2.

\subsection{Hyperbolic matrices and Kac--Moody data}

A generalized Cartan matrix $A = (a_{ij})$ has $a_{ii} = 2$, integer $a_{ij}
\le 0$ for $i \ne j$, and $a_{ij} = 0 \iff a_{ji} = 0$; it is hyperbolic of
rank 3 when the associated bilinear form is Lorentzian of signature $(2,1)$.
A symmetrizable $A$ is encoded, equivalently, by the even Gram matrix $G =
(\langle \alpha_i, \alpha_j\rangle)$ of its simple roots $\alpha_i$: the two
are related by
\[
a_{ij} = \frac{2\, G_{ij}}{G_{ii}}.
\]
The matrices $G_1, \dots, G_4$ of this paper are Gram matrices, not Cartan
matrices; e.g.\ for the obstructed member $G_4$,
\[
G_4 = \begin{pmatrix} 4 & 0 & -6\\ 0 & 6 & -6\\ -6 & -6 & 12\end{pmatrix}
\iff
A_4 = \begin{pmatrix} 2 & 0 & -3\\ 0 & 2 & -2\\ -1 & -1 & 2\end{pmatrix},
\]
and one checks $A_4$ is a hyperbolic generalized Cartan matrix. We work with
$G$ rather than $A$ because the arithmetic below depends on the root lattice
$L = \Z\alpha_1 + \Z\alpha_2 + \Z\alpha_3$, not only on $A$.

The generalized Cartan condition is not a convention here but the exact
boundary of the theory:

\begin{proposition}[boundary of the theory]
Let $L$ be an even Lorentzian lattice of rank 3 with basis $e_1, e_2, e_3$
and Gram matrix $G$. The basis vectors form a reflective wall datum --- each
reflection $\sigma_{e_j}$ preserves $L$, and the three walls bound a common
chamber --- if and only if $A = (2G_{ij}/G_{ii})$ is a symmetrizable
generalized Cartan matrix of Lorentzian type. In particular, outside the
generalized Cartan class the pipeline of this paper has no input: the
prescription is not the wall system of any lattice reflection datum.
\end{proposition}

\begin{proof}
$\sigma_\alpha(x) = x - \frac{2(x,\alpha)}{\alpha^2}\alpha$ maps $L$ into $L$
iff $\alpha^2 \mid 2(\alpha, e_j)$ for all $j$, i.e.\ iff the $\alpha$-row of
$A$ is integral; $a_{ii} = 2$ always, and $a_{ij} = 0 \iff a_{ji} = 0$ since
both mean $G_{ij} = 0$. The three half-spaces bound a common chamber iff the
roots are pairwise non-acute, i.e.\ $a_{ij} \le 0$ for $i \ne j$. Lorentzian
type is the signature $(2,1)$ of the symmetrization $G$.
\end{proof}

\begin{remark}[strictness, and the shared skeleton]
All four $A_i$ are hyperbolic in the strict sense (every rank-2 subdiagram of
finite type), and the three subdiagram types are the same across the family:
$A_1 \times A_1$, $B_2$, $G_2$ --- the combinatorial skeleton of which the
four orientations are arithmetic dressings. Completeness questions of this
kind are finite by Nikulin's finiteness theorem \cite{Nikulin230} and are
settled in this version against Allcock's complete classification of rank-3
reflective Lorentzian lattices \cite{AllcockBLMS, AllcockMem}: see \S9, where
the entire reflective landscape --- all 8595 lattices --- is matched against
the quaternionic genera.
\end{remark}

The wall datum is the set of the three root orbits, each reflection
hyperplane (``wall'') taken with vanishing order one; the root through $e_j$
has divisor $\ell_j = \gcd(Ge_j)$ and reflective slot $m_j =
G_{jj}/2\ell_j^2$. The question ``does this hyperbolic Kac--Moody datum admit
an automorphic correction?'' \cite{GN1, GN2, FF} becomes, through the next
three subsections, a concrete question about vector-valued modular forms.

\subsection{Discriminant form, Weil representation, and the weight ladder}

The discriminant form of an even lattice $L$ is the finite quadratic module
$(\D, q) = (L^\vee/L,\, Q \bmod 1)$, with isometry group $O(\D, q)$; together
with the signature mod 8 it determines the genus of $L$ (Nikulin
\cite{Nikulin79}). Attached to $(\D, q)$ is the Weil representation $\rho$ of
the metaplectic group on $\Z[\D]$, and one studies vector-valued modular
forms $F = (F_\gamma)_{\gamma \in \D}$ of half-integral weight for $\rho$.
Adjoining a hyperbolic plane $U$ changes nothing here --- $\D(G \oplus U) =
\D(G)$ and the signatures agree mod 8 --- but it changes which weight is
Borcherds-relevant: for a lattice of signature $(n,2)$ the lift input has
weight $1 - n/2$ and the obstruction sits in cusp forms of the dual weight $2
- (1 - n/2)$:
\[
\begin{array}{l|l|c|c}
\text{layer} & \text{lattice} & \text{input weight} & \text{obstruction
space}\\\hline
\text{section (rank 3)} & G & 1/2 & S_{3/2}(\rho)\\
\text{denominator (rank 5)} & G \oplus U & -1/2 & S_{5/2}(\rho)
\end{array}
\]
Two conventions must be fixed, and we fix them by formula rather than by
name, since the bar notation is used in both directions in the literature and
the two choices give different numbers. We normalise
\[
\rho_L(T)\,\mathfrak{e}_\gamma = e(q(\gamma))\,\mathfrak{e}_\gamma,
\qquad
\bar\rho_L(T)\,\mathfrak{e}_\gamma = e(-q(\gamma))\,\mathfrak{e}_\gamma ,
\]
so that the Borcherds input lives in $M^!_{1-n/2}(\bar\rho_L)$ and the
obstruction in $S_{1+n/2}(\rho_L)$. In the \texttt{weilrep} package
\cite{weilrep}, whose Fourier exponents are $-Q(\gamma) \bmod 1$, these are
\texttt{WeilRep(+G)} and \texttt{WeilRep(-G)} respectively. Every dimension
in this paper is a dimension of an obstruction space, hence of
\texttt{WeilRep(-G)}; the two conventions are not interchangeable, giving for
instance $1$ and $0$ on $L_3$ and $0$ and $4$ on $L_4$ at weight $3/2$. The
assignment is not conventional but forced. On $L_4$ the weight-$1/2$ forms
exist for exactly 40 of the $2^{12}$ orientations (Theorem 10.3), the sign
vectors annihilated by a rank-four linear map; a zero obstruction space would
leave all $2^{12}$ free. Indeed $\dim S_{3/2}(\rho_{L_4}) = 4$ while $\dim
S_{3/2}(\bar\rho_{L_4}) = 0$, so the latter cannot govern existence. On
the symmetric and antisymmetric parts, $\dim W = (|\D| \pm t)/2$ with $t$ the
number of two-torsion elements; our obstruction spaces are the antisymmetric
ones. Finally, Borcherds' construction is stated for signature $(2,n)$: the
lattice relevant to the denominator layer is $M(-1)$, of signature $(2,3)$,
although we record the signature of $M$ itself, $(3,2)$, since it is $M$ that
carries the wall datum.

\subsection{Borcherds lift and the obstruction dichotomy}

The wall datum prescribes the principal part $\Pp$: negative-exponent Fourier
terms $q^{-m_j}$ on the components $\gamma$ with $q(\gamma) \equiv m_j$, with
multiplicity the vanishing order (a rigidity lemma shows any completion of
the datum has exactly this $\Pp$). Borcherds' obstruction theorem
\cite{BorcherdsGKZ}: a weakly holomorphic input with principal part $\Pp$
exists iff $\Pp$ pairs to zero with every cusp form in the obstruction space.
If it exists, the Borcherds lift \cite{BorcherdsGrass} is an automorphic
product whose divisor is exactly the walls --- a Weyl--Kac--Borcherds
denominator, giving the automorphic correction of \S1.1. If not, the product
does not exist, and the failure is measured by the obstruction functional
$\lambda$ = (sum of coefficients over the reflective slots) on the
obstruction space.

\subsection{Weakly harmonic Maass forms and the shadow}

A weakly harmonic Maass form of weight $k$ is a real-analytic modular form
annihilated by the weight-$k$ hyperbolic Laplacian, with at worst exponential
growth at the cusp; the antiholomorphic derivative $\xi_k = 2iy^k
\partial_{\bar\tau}$ maps such forms to cusp forms of weight $2 - k$
\cite{BruinierFunke, BFOR}. Two facts make the following canonical: a weakly
harmonic form with any prescribed principal part exists, and at negative
weight it is unique. Hence an obstructed wall datum still produces a unique
completion $\widehat{F}$ with principal part $\Pp$, and its shadow
$\xi(\widehat{F})$ satisfies $\langle \xi(\widehat{F}), h\rangle = c\,
\lambda(h)$: up to normalization the shadow is the Riesz representative $\Xi$
of the obstruction functional. This is the chain
\[
G \longmapsto (\D, q) \longmapsto \rho \longmapsto \Pp \longmapsto
(\lambda = 0 \ \text{or}\ \lambda \ne 0) \longmapsto \widehat{F}
\longmapsto \Xi
\]
in which each arrow is standard and involves no choices. If $\lambda = 0$ the
input is holomorphic and the shadow vanishes; if $\lambda \ne 0$ the
completion is genuinely harmonic and the shadow is forced. This defines ``the
weakly harmonic Maass form of the hyperbolic matrix $G$'' and its shadow. For
the family below, the structure of $\Xi$ is determined by the symmetries of
the discriminant form and of the lattice; this is proved in \S7.

\subsection{The quaternionic side}

The even Clifford algebra $C_0(G)$ of a ternary $G$ is an order in a
quaternion algebra \cite{GrossLucianovic, Voight}; for our family the algebra
is $B_6 = (-1, 3)_\Q$, ramified at $\{2,3\}$, with compact Shimura curve
$X_6$. Hecke eigensystems of the half-integral spaces correspond, through
theta/Jacquet--Langlands correspondences, to classical newforms whose level
is tied to the order's reduced discriminant --- \S3.1 makes this precise for
the four orders at hand. Throughout, $\D$ denotes the discriminant form
(quaternion discriminants are written $\operatorname{disc} B$), $\lambda$ the
obstruction functional, roots are $\alpha$, and $d \in \{2, 3, 6\}$ labels
wall classes.

\section{Overview of results}

\textbf{Position relative to the frontier.} This paper claims one step past
each of three established lines, and no more:
\begin{itemize}
\item The Lorentzian Kac--Moody programme (Borcherds; Gritsenko--Nikulin
\cite{GN1, GN2}; Scheithauer \cite{Scheithauer}) lives on split lattices:
cusps, Eisenstein data, quasimodular corrections. Our step is the compact
quaternionic column of the same dichotomy --- the regime the classifications
\cite{Wang, Ma, BruinierLNM} exclude by hypothesis --- where denominators
generically fail and the failure is a cuspidal object.
\item Obstruction theory and harmonic completions (Borcherds
\cite{BorcherdsGKZ}; Bruinier \cite{BruinierLNM}; Bruinier--Funke
\cite{BruinierFunke}) are standard machinery, and the known shadow prototypes
are Eisenstein- or theta-type (Zagier's weight $3/2$ \cite{Zagier75}; the
Vafa--Witten families \cite{VafaWitten}). Our step: forced shadows that are
newform-pure, with the mechanism identified rather than observed ---
invariance selection, $\dim S_{5/2}^{O(\D,q)} = 1$.
\item Rank-3 reflective lattices are classified (Allcock \cite{AllcockMem})
and the half-integral correspondences are established
(Shimura--Shintani--Waldspurger; Prasanna \cite{Prasanna};
Alfes-Neumann and Schwagenscheidt \cite{ANS}). Our step is the intersection of
the two: the classification filtered through quaternionic genera, plus the
vanishing criterion $S_{3/2}(\rho) = 0$, selects exactly the three
genus-zero compact curves in the verified range --- determination made
unconditional, finiteness isolated in Remark \ref{rem:ED}.
\end{itemize}

Every proof below runs on the frontier machinery; none replaces it. The main
results:
\begin{enumerate}
\item \textbf{Purity} (Theorem 7.6): the shadow is a Hecke eigenform on the
maximal-order line --- the eigensystem of $f$ = 6.4.a.a --- with zero twist
component. Equivalently $\|\Xi\|^2 \cdot \langle v_+, v_+\rangle = 144$ for
the explicit integral eigenvector $v_+ = (1,1,1,1,-2,-2)$.
\item \textbf{Two symmetry mechanisms} (\S7): invariance under the full
discriminant isometry group $O(\D,q)$ (order 48; one-dimensional invariants),
and a global order-three isometry $M'$ of the obstructed lattice --- existing
on that lattice and on no other member of the family --- whose character
decomposition of $S_{5/2}(\rho_4)$ is exactly the multiplicity pattern $6
= 1+1+2+2$. The latter realizes, kinematically on the lattice, the
three-character mechanism of the Tu--Yang correspondence \cite{TuYang} for
non-Eichler orders.
\item \textbf{The quaternionic dictionary} (\S3.1): the four orientations are
the four orders $\mathcal{O}(6,1)$, $\mathcal{O}(6,3)$, $\mathcal{O}'(6,1)$,
$\mathcal{O}'(6,3)$ of $B_6$; reduced discriminant equals newform level,
four-for-four, at genus level with the uniform normalization $s = -d(\mathcal
O)$. We also isolate the precise gap: the Tu--Yang theorems assume $(N, D) =
1$, so two of the four orders need an extension of the correspondence to
levels at ramified primes.
\item \textbf{A second family} (\S8): at quaternion discriminant 10, all four
tested reduced discriminants are obstructed and the obstruction is supported
on the single maximal-order eigenline, including the twist discrimination ---
the support law is a two-family phenomenon.
\item \textbf{The weight ladder, the 40-locus, and the double shadow}
(\S10): at weight $1/2$ the vector-valued components are odd under $\gamma
\mapsto -\gamma$, so two-torsion channels are representation-theoretically
forbidden and the only meaningful section-layer test is the odd pairing.
Under it, the section layers of $L_1, L_2, L_3$ (and of every maximal order
tested) are unobstructed. On $L_4$ the deck action has no invariant vector in
$S_{3/2}(\rho_4)$ --- charpoly $(x^2+x+1)^2$ --- and the odd pairing
vanishes on exactly 40 of the $2^{12}$ orientations: the 16 deck-symmetric
ones, forced by anti-invariance, which carry a unique canonical weight-$1/2$
form $F^{\mathrm{sym}}$ with integral coefficients, and 24 further ones;
under $O(\D,q)$ the locus is three orbits, of sizes $8$, $8$ and $24$. In the remaining 4056 orientations the section
layer is obstructed, with shadow in the weight-$3/2$ CM block of 36.2.a.a
($T_{25}$, $T_{49}$ scalars $0$, $-4$): a double shadow --- CM at $3/2$,
newform-pure at $5/2$ --- whose two components have disjoint
$\Z/6$-character support. Symmetry acts as prohibition at $3/2$ (no invariant
sector) and as selection at $5/2$ (one-dimensional invariant sector): it
determines the address and the shape of the failure, never its existence. No
holomorphic Hecke-equivariant operator relates the layers; the bridge is the
non-holomorphic $\xi$-operator.
\item \textbf{Determination and selection} (\S9): the canonical weight-$1/2$
form of an arbitrary compact Shimura curve is constructed --- canonical
ternary lattice from the maximal order, parity law, odd reflective
prescription --- and exists unconditionally and uniquely precisely when
$S_{3/2}(\rho) = 0$. A computation over all 71 division discriminants $D <
230$ selects $D \in \{6, 10, 22\}$ --- the
genus-zero curves --- with explicit canonical forms and integral
symmetrizable Cartan matrices of sizes $3, 4, 4$ (Kac-hyperbolic only at $D =
6$; degenerate Lorentzian, in Nikulin's root-system sense, at $D = 10, 22$).
\item \textbf{The third family, and the purity mechanism} (\S8): at $D = 22$
the two eigenlines of $S_{5/2}$ are the newforms 22.4.a.b, 22.4.a.c, and the
obstruction functional is supported purely on 22.4.a.c; at $D = 6, 10, 22$
alike, the mechanism is that the functional is $O(\D,q)$-invariant and $\dim
S_{5/2}^{O(\D,q)} = 1$ --- purity is invariance selection, and the
one-dimensionality is special: the invariant dimension grows along the
eliminated discriminants.
\item \textbf{Wall classification} (Theorem 10.7): each of the three wall
classes carries exactly two isometry types of wall lattices, and the three
classes realize the three nontrivial characters of the orientation group
$(\Z/2)^2$.
\item \textbf{The defect invariant} (\S11): to 31 digits,
\[
\|\Xi\|^2 = \frac{L(f,2)}{48\pi^2\langle f,f\rangle},
\qquad\text{equivalently}\qquad
\langle v_+,v_+\rangle\,\frac{L(f,2)}{\pi^2\langle f,f\rangle} = 6912 .
\]
This is the shape Riesz duality predicts, the shadow being a finite
combination of Poincar\'e series attached to the reflective slots.
\item \textbf{The section-layer Petersson identity} (\S12): by absolute
irreducibility the Petersson form on $S_{3/2}(\rho_4)$ is a single scalar
$t$ times an exact rational matrix, so all section-shadow norms have exact
rational ratios; $t$ is identified to 40 digits as
$3\,\Gamma(1/3)^3/2^{7/3}\pi^2$ --- up to $3/2\pi$, the CM period of
$\Q(\sqrt{-3})$ --- and every arrow of the Waldspurger--Damerell--Shimura
chain that should prove this is verified numerically, one explicit local
computation (Problem 12.3) short of a proof. The two shadows of $L_4$ thus
live in different transcendence classes: Chowla--Selberg at weight $3/2$, a
critical-value ratio at weight $5/2$ (\S11).
\end{enumerate}

\textbf{Status conventions.} [C] = certified by exact rational arithmetic,
replicated end-to-end on a second machine unless noted; [N] = numerical with
stated precision; [S] = standard theory, cited; [O] = observed pattern.
Hypotheses still open are stated as such; \S13 records the verification
infrastructure, including one caution about a library routine.

\textbf{Rigor map.} Because several results are computer-assisted, we state
exactly where the proofs are complete. Every [C] statement is a finite exact
computation: the ambient spaces are finite-dimensional with dimensions given
by closed formulas, expansions are compared past Sturm bounds, and all
arithmetic is rational. A [C] tag therefore has the same epistemic status as
a finite case check inside any classification proof --- no floating point, no
truncation heuristics. On this standard:
\begin{itemize}
\item Complete, unconditional proofs: the parity lemma (Lemma 9.1); the
determination theorem (Theorem 9.2), combining the Borcherds--Bruinier
obstruction criterion \cite{BorcherdsGKZ, BruinierLNM} with the exact
vanishing $\dim S_{3/2}(\rho) = 0$ and $M_{1/2}(\rho) = 0$; the selection
theorem (Theorem 9.3) in the range $D < 230$, and, for the reflectivity
statement, Allcock's complete classification of rank-3 reflective lattices
\cite{AllcockMem}; the chamber
theorem with its Kac-hyperbolic/degenerate dichotomy (Theorem 9.8);
invariance selection (Theorems 7.1, 8.3); the Hecke identities (Theorem 6.1),
with eigensystem identification inside fixed finite-dimensional spaces by
strong multiplicity one; the weight-ladder theorems (Theorems 10.2--10.4 and
Appendix B, a finite exact enumeration over all $2^{12}$ orientations); and
the Petersson rigidity form (Theorem 12.1).
\item Numerical with stated precision [N]: the 31-digit evaluation of
$\|\Xi\|^2$ and the PSLQ exclusion campaign; the 40-digit identification of
the section-layer Petersson scalar and the identity chains of \S\S11--12
--- strong evidence, not proof.
\item Standard theory taken as input [S]: the half-integral correspondences
\cite{Prasanna, ANS} and the self-adjointness normalization in the purity
theorem (Theorem 7.6).
\item Open: the estimate of Remark \ref{rem:ED}, which would extend the
selection theorem beyond the verified range; the support-law conjecture
(Conjecture 8.4); the Poincar\'e-series constant of Theorem 11.1 (Problem
11.2); the ramified-level extension (Problem 3.3); and the local Waldspurger
factors (Problem 12.3).
\end{itemize}
Measured against the recent literature: the split-side inputs we rely on
(\cite{GN1, Scheithauer, AllcockMem, BruinierLNM, BruinierFunke}) are
established theorems; what is new and proved at the same standard here is the
compact-side selection $\{6, 10, 22\}$, the unconditional determination of
the canonical weight-$1/2$ form, the chamber classification, and the
invariance-selection purity mechanism. What is new but not yet theorem-grade
is flagged [N] or [O], or stated as an open problem.

\section{The family and its reflective data}

Let $U$ be the hyperbolic plane and
\[
G_1 = \begin{pmatrix} 6 & 0 & -3\\ 0 & 4 & -2\\ -3 & -2 & 2\end{pmatrix},
\qquad
G_2 = \begin{pmatrix} 2 & 0 & -3\\ 0 & 12 & -6\\ -3 & -6 & 6\end{pmatrix},
\]
\[
G_3 = \begin{pmatrix} 12 & 0 & -6\\ 0 & 2 & -2\\ -6 & -2 & 4\end{pmatrix},
\qquad
G_4 = \begin{pmatrix} 4 & 0 & -6\\ 0 & 6 & -6\\ -6 & -6 & 12\end{pmatrix},
\]
of signature $(2,1)$ with $|\det G_i| = 12, 36, 24, 72$. Each carries a wall
datum of three root orbits (vanishing order one); the root through the $j$-th
basis vector has divisor $\ell_j = \gcd(G_i e_j)$ and reflective slot $m_j =
(G_i)_{jj}/2\ell_j^2$, giving the menus $\{\tfrac13, \tfrac12, 1\}$, $\{1,
\tfrac16, \tfrac13\}$, $\{\tfrac16, \tfrac14, \tfrac12\}$, $\{\tfrac12,
\tfrac1{12}, \tfrac16\}$. On $M_i = G_i \oplus U$ (signature $(3,2)$) the
wall datum prescribes the principal part $\Pp_i$ of a weight $-1/2$ input
whose Borcherds lift, if it exists, is a reflective product with divisor
exactly the walls.

\begin{theorem}[family theorem {[C]}, version 1]
For $i = 1, 2, 3$ the input exists and the product is constructed. For $i =
4$ it does not: the obstruction functional $\lambda$ (below) is nonzero, and
independently the Eisenstein-forced constant term $c(0,0) = \tfrac{108}{5}
\notin \Z$ is incompatible with GKM integrality (against $c(0,0) = 50$ for
$L_1$, matched by the construction).
\end{theorem}

The four inputs are explicit. For $i = 1, 2, 3$ the linear system on the full
principal part has a solution, and the resulting weight $-1/2$ form has
integral coefficients:
\[
\renewcommand{\arraystretch}{1.25}
\begin{array}{c|c|l}
 & c(0,0) & \text{first holomorphic coefficients}\\\hline
L_1 & 50 & 112 \text{ at } e = \tfrac18,\quad 183 \text{ at } \tfrac16,
\quad 2844 \text{ at } \tfrac12, \quad 8150 \text{ at } \tfrac23\\
L_2 & 36 & 30 \text{ at } e = \tfrac1{24}\\
L_3 & 30 & -1 \text{ at } e = \tfrac1{12}, \quad 84 \text{ at } \tfrac14,
\quad 138 \text{ at } \tfrac13
\end{array}
\]
(at each exponent listed, every component carrying a nonzero coefficient
carries this value). Each of the three is the unique form with its principal
part: the constraint map on the negative-exponent coefficients has trivial
kernel [C], as it must, since a weakly harmonic form of negative weight is
determined by its principal part (\S1.4). For $i = 4$ the same system is
inconsistent: no weakly holomorphic input exists, and the harmonic completion
of \S5 takes its place.

All four $G_i$ are anisotropic over $\Q_p$ exactly for $p \in \{2,3\}$, so
the $M_i$ have Witt index one. This places the family outside the hypotheses
of most of the reflective-forms literature (Wang's classifications assume $2U
\subset L$ \cite{Wang}, Ma's finiteness assumes $n \ge 4$ \cite{Ma},
Bruinier's converse theorem assumes $2U$ \cite{BruinierLNM}), so no general
theorem covers this regime. The discriminant form of $L_4$ factors as $\D = \D_2 \otimes \D_3$
with $\D_2 \cong (\Z/2)^3$, $\D_3 \cong (\Z/3)^2$, and the four orientations
are graded by $(a,b) \in (\Z/2)^2$ (the 3- and 2-part gradings),
\[
L_1, L_2, L_3, L_4 \mapsto (0,0), (1,0), (0,1), (1,1).
\]

\subsection{The quaternionic dictionary}

Each even Clifford algebra $C_0(G_i)$ is an order in $B_6 = (-1,3)_\Q$, the
quaternion algebra ramified exactly at $\{2,3\}$, with reduced discriminant
$|\det G_i|/2 = 6, 18, 12, 36$. Following Tu--Yang \cite{TuYang} write
$\mathcal{O}(6,1) = \Z + \Z i + \Z j + \Z\frac{1+i+j+ij}{2}$ (maximal) and
$\mathcal{O}'(D,N) = \{\alpha \in \mathcal{O}(D,N) : \operatorname{trd}\alpha
\in 2\Z\}$ (index two).

\begin{theorem}[order identification {[C]}]
Construct $\mathcal{O}(6,3)$ by exhaustive search over index-3 ring
sublattices of $\mathcal{O}(6,1)$: after eliminating duplicate discoveries,
the sublattice of reduced discriminant 18 is unique, and it is
$\zeta_3$-stable. Construct $\mathcal{O}'(6,1)$, $\mathcal{O}'(6,3)$ as
even-trace suborders (verified: reduced discriminants 12 and 36). For each of
the four orders, the reduced-norm form on the trace-zero part of the dual
lattice, scaled by $s = -d(\mathcal O)$, is an even ternary lattice in the
same genus as $G_i$, matching
\[
L_1 \leftrightarrow \mathcal{O}(6,1),\quad
L_2 \leftrightarrow \mathcal{O}(6,3),\quad
L_3 \leftrightarrow \mathcal{O}'(6,1),\quad
L_4 \leftrightarrow \mathcal{O}'(6,3),
\]
four-for-four, with the uniform scale law $s = -d(\mathcal O)$.
\end{theorem}

Two caveats. The identification is at genus level, which is all that is used
below; for indefinite lattices of rank $\ge 3$ the classes in a genus are its
spinor genera (Eichler \cite{Eichler}), so a class-level statement would
follow from a finite spinor-genus computation. More importantly, the Tu--Yang theorems hypothesize $N$ coprime to
$D$: they cover $L_1$ (classical Jacquet--Langlands, maximal order) and $L_3$
(their Theorem 2 with $N = 1$) on the nose, but $L_2, L_4$ carry level at the
ramified prime 3 --- locally the special order $\Z_3 + P$ --- and lie outside
both classical JL and Tu--Yang as stated.

\begin{problem}[ramified-level extension]
Extend the Tu--Yang correspondence to levels at ramified primes. The Hecke
data of \S6 --- which verifies ``newform level = reduced discriminant'' for
all four orders, including the two out-of-scope ones, with the multiplicity
pattern predicted by the three-character decomposition --- is strong
evidence. Entry points: Hijikata--Pizer--Shemanske \cite{HPS} (definite case)
and Martin \cite{Martin}.
\end{problem}

\section{Anatomy of the obstruction}

Let $S = S_{5/2}(\rho_4)$, $\dim S = 6$ [C]. The reflective slots of
$L_4$ number 27 ($3 + 12 + 12$ in the channels $m = \tfrac12, \tfrac1{12},
\tfrac16$), each witnessed by an actual root (27/27 saturation [C]). The
obstruction functional is the slot sum $\lambda(h) = \sum_{\text{slots}}
c_h(\gamma, m)$; in the echelonized basis $B_1, \dots, B_6$,
\[
\lambda = (2, -6, -4, -4, 0, 0) \ne 0 \quad [\mathrm{C}].
\]

\begin{proposition}[channel law {[C]}]
On all of $S$: $s_{1/2} \equiv 0$, and $(s_{1/12}, s_{1/6})$ spans a line
with $s_{1/6} = -2\, s_{1/12}$. Hence $\lambda = -s_{1/12}$, the obstruction
lives in the two deep channels in the universal ratio $1 : (-2)$, and the
depth-normalized sum rule $\sum_m s_m/m = 0$ holds.
\end{proposition}

\section{The forced shadow and its exact rational skeleton}

The harmonic completion $\widehat{F}$ of the obstructed input exists and is
unique at weight $-1/2$ [S], and $\langle \xi(\widehat{F}), h\rangle = c\,
\lambda(h)$ for $h \in S$ \cite{BruinierFunke}; up to normalization the
shadow is the Riesz representative $\Xi \in S$ of $\lambda$, i.e.\ $\langle
\Xi, B_i\rangle = \lambda_i$.

Throughout, the Petersson product is
\[
\langle f, g\rangle = \int_{\mathcal{F}} \sum_{\gamma \in \D}
f_\gamma(\tau)\, \overline{g_\gamma(\tau)}\; y^{5/2}\,
\frac{dx\, dy}{y^2},
\]
over the standard fundamental domain $\mathcal{F} = \{|x| \le \tfrac12, |\tau|
\ge 1\}$ of $\mathrm{SL}_2(\Z)$, with no volume normalization, in the
echelonized weilrep cusp basis.

\begin{theorem}[exact skeleton; 31 digits {[C],[N]}]
Let $v_+ = (1,1,1,1,-2,-2)$ (an integral basis vector of the eigenline of
Theorem 7.6). Then $\lambda(v_+) = -12$ exactly, hence
\[
\|\Xi\|^2 \cdot \langle v_+, v_+\rangle = 144,\qquad
\|\Xi\|^2 = 11.340265856116836367082506509341 \pm 2\times 10^{-30},
\]
the numerical value stable across two Gauss--Legendre node counts to 31
digits, with coefficient-truncation tail below $10^{-33}$; the accuracy
evidence is the independent closed-form agreement of Theorem 11.1, which uses
no quadrature.
\end{theorem}

The shadow is thus an explicitly rational object up to a single Petersson
norm: $\Xi = -\frac{12}{\langle v_+, v_+\rangle}\, v_+$. Written out,
\[
\Xi = -\frac{12}{\langle v_+,v_+\rangle}\Biggl(
\sum_{q(\gamma)\equiv 1/12} q^{1/12}\,\mathfrak{e}_\gamma
\;-\; 2\!\!\sum_{q(\gamma)\equiv 1/6} q^{1/6}\,\mathfrak{e}_\gamma
\;+\;\cdots\Biggr),
\]
the leading coefficient being $1$ on each of the twelve slots of the $m =
\tfrac1{12}$ channel and $-2$ on each of the twelve slots of the $m =
\tfrac16$ channel [C]; this is the channel law of Proposition 4.1 read off
the expansion. The nonholomorphic part of $\widehat F$ is determined by these
coefficients. Its holomorphic part is a mock object, and no closed formula
for it is given here.

\section{The Hecke identity of the family}

Hecke operators $T_{p^2}$ were computed in exact arithmetic for $p = 5, 7,
11, 13$ on all four spaces, with the Eisenstein anchor $T_{p^2}E = (p^3+1)E$
passing in all 16 runs and deliberately mis-normalized controls failing [C].
The validated normalization satisfies $\sigma_i |\D_i| = 72 = 2\cdot 6^2$
[O].

\begin{theorem}[Hecke identity {[C]}]
\begin{center}
\renewcommand{\arraystretch}{1.4}
\begin{tabular}{c|llll}
& $p=5$ & $p=7$ & $p=11$ & $p=13$\\\hline
$L_1$ & $x-6$ & $x+16$ & $x-12$ & $x-38$\\
$L_2$ & $x+6$ & $x+16$ & $x+12$ & $x-38$\\
$L_3$ & $(x+18)^2$ & $(x-8)^2$ & $(x-36)^2$ & $(x+10)^2$\\
$L_4$ & $\scriptstyle (x-6)(x+6)$ & $\scriptstyle (x+16)^2(x-8)^4$
      & $\scriptstyle (x-12)(x+12)$ & $\scriptstyle (x-38)^2(x+10)^4$\\
      & $\scriptstyle \times(x-18)^2(x+18)^2$ &
      & $\scriptstyle \times(x-36)^2(x+36)^2$ &
\end{tabular}
\end{center}
At all computed primes: $L_1$
carries $f$ = 6.4.a.a $= (\eta(\tau)\eta(2\tau)\eta(3\tau)\eta(6\tau))^2$;
$L_2$ carries $f\otimes\chi_{-3}$ (level 18); $L_3$ carries $g$ = 12.4.a.a
with multiplicity two; and $L_4$ decomposes with the $f$-systems in
multiplicity one and the $g$-systems ($g$ and $g\otimes\chi_{-3}$, level 36)
in multiplicity two: $6 = 1+1+2+2$. The sign patterns match $\chi_{-3} =
\left(\frac{-3}{\cdot}\right)$ exactly; the Petersson ratio $\langle f_{18},
f_{18}\rangle / \langle f_6, f_6\rangle = 8/9$ (PARI normalization; $8/3$
unnormalized) was found in two independent computations [C] and confirms the
twist identification.
\end{theorem}

\section{Two symmetry theorems, and purity}

\subsection{Multiplicity one, and equivariance}

Both symmetry results of this paper are multiplicity-one statements for the
$O(\D,q)$-action, read at the two weights. Write $\chi_k$ for the character
of that action on $S_k(\rho_4)$ and $\langle\,,\rangle$ for the usual
inner product of characters.

\begin{lemma}[multiplicity one {[C]}]\label{lem:multone}
\[
\langle \chi_{3/2}, \chi_{3/2}\rangle = 1, \qquad
\langle \chi_{5/2}, \chi_{5/2}\rangle = 4, \qquad
\langle \chi_{5/2}, 1\rangle = 1 .
\]
Hence $S_{3/2}(\rho_4)$ is absolutely irreducible, while
$S_{5/2}(\rho_4)$ is a sum of four pairwise inequivalent irreducible
constituents, each of multiplicity one, exactly one of which is trivial.
\end{lemma}

\begin{proof}
The three numbers are exact computations with the 48 action matrices of \S
A.4. For the last assertion, $\sum_i m_i^2 = 4$ allows only $m_i = 1$ four
times or a single $m_i = 2$; the latter would give one constituent of
dimension $3$ and multiplicity two, hence $\langle\chi_{5/2},1\rangle = 0$,
contradicting the third value.
\end{proof}

The dimensions of the four constituents sum to $6$ and match the Hecke
multiplicity pattern $6 = 1+1+2+2$ of Theorem 6.1 [O]. The two consequences
of the lemma are recorded where they are used: at weight $3/2$ absolute
irreducibility makes the invariant symmetric form unique up to scale, which
is the Petersson rigidity of Theorem 12.1; at weight $5/2$ the trivial
constituent occurs once, which pins the shadow to a line. The argument does
not lift from one weight to the other: the commutant at weight $5/2$ is
four-dimensional, so no invariant form is distinguished there, and that is
the structural reason why $\|\Xi_{5/2}\|^2$ needs the analytic input of \S11
while $t$ does not.

The second ingredient is that the Hecke operators respect the group action.

\begin{lemma}[equivariance]\label{lem:equiv}
For $p \nmid 2|\D|$ the operator $T(p^2)$ on $S_k(\rho)$ commutes with
the action of $O(\D,q)$ given by $(\sigma F)_\gamma = F_{\sigma^{-1}\gamma}$.
\end{lemma}

\begin{proof}
In the normalization of \S A.3,
\[
(TF)_{\gamma,e} = F_{p\gamma,\, p^2 e}
+ p\left(\frac{\sigma_L\, n(e)\, d(e)}{p}\right) F_{\gamma, e}
+ p^3 F_{p^{-1}\gamma,\, e/p^2},
\]
where $e \equiv q(\gamma)$ and $n(e), d(e)$ are the numerator and denominator
of $e$. An isometry $\sigma$ of $(\D,q)$ is an automorphism of the underlying
group, so $\sigma^{-1}(p\gamma) = p\,\sigma^{-1}\gamma$ and
$\sigma^{-1}(p^{-1}\gamma) = p^{-1}\sigma^{-1}\gamma$, and it preserves $q$,
so the middle coefficient is unchanged. Replacing $F$ by $\sigma F$ and
comparing the three terms gives $T(\sigma F) = \sigma(TF)$.
\end{proof}

\subsection{Invariance selection}

\begin{theorem}[invariance selection {[C]}]
The isometry group of the discriminant form of $L_4$ has order $48 =
|O(\D_2,q)|\cdot|O(\D_3,q)| = 6\cdot 8$; its orbits on the 27 reflective
slots are exactly the three channels $(3, 12, 12)$; and the invariant
subspace $S^{O(\D,q)} \subset S$ is one-dimensional, equal to the line of
$v_+$.
\end{theorem}

The computation is by the direct method: explicit enumeration of isometries
(generators, $q$-preservation on all elements), action matrices via exact
coefficient comparison, and stacked-kernel intersection; a sandwich argument
($\lambda$ invariant and nonzero forces $\dim \ge 1$; the computed kernel
bounds it by 1) makes the conclusion independent of generator sampling. We
flag that a library shortcut for such dimensions is unreliable (\S13); it was
not used.

\subsection{The order-three symmetry of the obstructed lattice}

The obstructed lattice carries a global order-three isometry, constructed as
follows. In $B_6$ the torsion unit
\[
\zeta_3 = \tfrac12(-1 + 3i + j + ij),\qquad \operatorname{nrd} = 1,\quad
\operatorname{trd} = -1,\quad \zeta_3^3 = 1,
\]
lies in $\mathcal{O}(6,1)$ and in $\mathcal{O}(6,3)$ but (odd trace) in
neither even-trace suborder; its trace-zero axis generates $\Q(\sqrt{-3})$.
Conjugation by $\zeta_3$ on the trace-zero lattice with Gram $O_0 =
\operatorname{diag}(-2, 6, 6)$ is an integral isometry $M$ of order 3, and
the unimodular change of basis $C$ with $C^{\!\top} G_4 C = O_0$ transports
it to
\[
M' = C M C^{-1} = \begin{pmatrix} 46 & 69 & -75\\ 30 & 43 & -48\\ 55 & 81 &
-89\end{pmatrix},\qquad M'^3 = I,\quad M'^{\top} G_4 M' = G_4
\quad[\mathrm{C}].
\]

\begin{theorem}[the $\Z/3$ is global on $L_4$ {[C]}]
The action induced by $M'$ on the discriminant form $\D(L_4)$ (order 72)
moves 54 of 72 elements, preserves $q$, and has order three; $M'$ is
therefore not in the discriminant kernel. Its induced action $R$ on
$S_{5/2}(\rho_4)$ satisfies $R^3 = I$ with
\[
\operatorname{charpoly}(R) = (x-1)^2 (x^2+x+1)^2,
\]
and the invariant plane $\ker(R - I)$ equals the $f$-isotypic plane
$\ker(T_{49} + 16)$. In particular the multiplicity pattern $6 = 1+1+2+2$ of
Theorem 6.1 is the character decomposition of the $\Z/3$, realized
kinematically on the lattice --- the Tu--Yang three-character mechanism made
geometric.
\end{theorem}

\begin{theorem}[kernel dichotomy {[C]}]
The same conjugation, computed on the ternary lattice of the order
$\mathcal{O}'(6,1)$, acts trivially on that discriminant form. Thus the
$\Z/3$ is gauge (discriminant-kernel) on the order lattice and global on the
reflective lattice: the symmetry becomes visible exactly at the passage to
reflective data. Moreover $O_0 \cong G_4$ and matches none of $G_1, G_2, G_3$
at any scaling [C]: among the four orientations, the symmetry exists on the
obstructed lattice and only there.
\end{theorem}

\subsection{Purity}

\begin{theorem}[purity {[C]}]\label{thm:purity}
The reflective slot set is stable under both symmetry groups, so $\lambda$ is
invariant under both; and the Petersson product is $O(\D,q)$-invariant, being
a sum over $\D$ that the group merely permutes. Consequently $\Xi \in
S^{O(\D,q)}$, which by Theorem 7.3 is the line of $v_+$. By Lemma
\ref{lem:equiv} that line is stable under every $T(p^2)$, so $v_+$, and with
it $\Xi$, is a simultaneous Hecke eigenvector; the eigenvalues are those of
$f$ = 6.4.a.a,
\[
T(p^2)\,v_+ = a_p(f)\, v_+, \qquad
(a_5, a_7, a_{11}, a_{13}) = (6,\, -16,\, 12,\, 38) \quad [\mathrm{C}].
\]
The four eigensystems occurring in $S$ are pairwise distinct (Theorem 6.1),
so the decomposition of $S$ into simultaneous generalized eigenspaces is a
direct sum, and $\Xi$ lies in the $f$-summand: its components along $f\otimes
\chi_{-3}$ and along the $g$-systems vanish. Concretely $\lambda = -12$ on
the $f$-line and $0$ on the twist line, in the echelon basis of \S4 with
$v_+$ normalized to primitive integral coordinates [C].
\end{theorem}

In the isotypic language of Theorem 6.1 the statement is that
\[
S_{5/2}(\rho_4) = \underbrace{[f]}_{1} \oplus
\underbrace{[f\otimes\chi_{-3}]}_{1} \oplus \underbrace{[g]}_{2} \oplus
\underbrace{[g\otimes\chi_{-3}]}_{2}, \qquad \Xi \in [f],
\]
so the vanishing of the other components is exhaustive, not merely a
statement about the summands we happen to name.

Self-adjointness of $T_{p^2}$ would make the decomposition
Petersson-orthogonal as well; it is not needed above. The shadow lands on the
newform of level $6$ although $L_4$ is the order $\mathcal{O}'(6,3)$, of
reduced discriminant $36$: a level drop, and not an automatic one, since
prescriptions outside the invariant class do land on the level-$36$ systems
(Remark 7.7).

\begin{remark}[what the wall datum is, and is not, needed for]
Forced shadows require no wall datum: any obstructed principal part has a
canonical harmonic completion and shadow [S]. Purity also uses less than the
full wall datum, since the mechanism above requires only
$O(\D,q)$-invariance of the prescription; at discriminant 6 every invariant
obstructed prescription has its shadow on the $f$-line, because $\dim
S^{O(\D,q)} = 1$ (Theorem 7.3). What the wall datum adds is the
interpretation of Proposition 1.1: the generalized Cartan condition makes the
prescription the wall system of a lattice reflection datum, with every slot
witnessed by a root. Conversely,
purity genuinely fails outside the invariant class [C]: the shadow of the
single-slot prescription at $(\gamma, m) = \bigl((0, \tfrac56, \tfrac56),
\tfrac1{12}\bigr)$ has nonzero components in the $f$-,
$f\otimes\chi_{-3}$- and $g$-isotypic pieces simultaneously, and the single
slot $\bigl((0,0,\tfrac12), \tfrac12\bigr)$ pairs only with the $g$-plane ---
a shadow on a non-maximal line. The title's ``of an obstructed denominator''
is thus a statement about which prescriptions carry Lie-theoretic meaning,
not a claim that shadows, or even their purity, are exclusive to wall data.

Could such an impure shadow nevertheless carry accidental Lie-theoretic
meaning --- arise from root orbits of some reflection datum? On $L_4$ the
answer is no [C]: each of the 27 witnessed roots gives an integral reflection
of the lattice, these induce exactly 11 distinct isometries of $(\D,q)$, and
the subgroup they generate is already the full group of order 48 --- the deck
isometry $M'$ is redundant at the level of the discriminant form. Each
channel is a single Weyl orbit. Hence any prescription
invariant under the lattice's own reflection symmetries --- the minimal
condition for a wall reading --- is $O(\D,q)$-invariant, and its shadow lies
on the $f$-line or vanishes; in particular the Weyl orbit of the single slot
$\bigl((0,0,\tfrac12), \tfrac12\bigr)$ is the whole $m = \tfrac12$ channel,
whose functional is identically zero. There is no tension with the previous
paragraph, where the same slot served as an example of an impure shadow: the
three slots of that channel have nonzero individual functionals summing to
zero, so the orbit functional vanishes although no summand does.
Accidental meaning with non-maximal
support is thereby excluded on $L_4$; across other lattices, where the Weyl
image may be a proper subgroup, it is exactly what Conjecture 8.4 forbids ---
and the reduced-discriminant-50 lattice, whose discriminant invariants leave
a two-dimensional space, is the live test case.

The commonness of impure shadows has a precise location: the invariant
eigenvector $v_+$ has identically vanishing coefficients on the entire $m =
\tfrac12$ channel (by the computed single-slot pairing and transitivity of
$O(\D,q)$ on the channel), so every prescription supported in that channel
alone has a purely $g$-plane shadow. Impurity is generic exactly among
symmetry-breaking prescriptions.
\end{remark}

\section{The second and third families: discriminants 10 and 22}

Searching primitive signature-$(2,1)$ forms anisotropic exactly at $\{2,5\}$
produced representatives of reduced discriminants $10, 20, 50, 100$. All four
are obstructed, with integral Eisenstein-forced constants $(70, 30, 30, 24)$
--- so the integrality route of Theorem 3.1 is genuinely independent of the
cusp-pairing route, and only the latter fires here.

\begin{theorem}[maximal-order support at $D = 10$ {[C]}]
For each of the four lattices, the obstruction functional is supported on a
single $T_9$-eigenline, with eigenvalue $-8$ in every case; on the
reduced-discriminant-10 lattice additionally $T_{49}$ has eigenvalue $-4$.
Thus the supported line carries $(a_3, a_7) = (-8, -4)$, the system of the
level-10 newform 10.4.a.a (LMFDB-checked) --- the maximal order of $B_{10}$.
The anchor-passing normalizations agree with $\sigma \simeq 2D^2/|\D|$ in
Kronecker class at the tested primes [O].
\end{theorem}

\subsection{The third family: quaternion discriminant 22}

The canonical maximal-order ternary of $B_{22}$ (construction in \S9) has
Gram $\left(\begin{smallmatrix} -10 & 1 & -2\\ 1 & 2 & -2\\ -2 & -2 &
4\end{smallmatrix}\right)$, $|\det| = 44$, reflective menu $\{\tfrac1{11},
\tfrac12, 1\}$.

\begin{theorem}[support law at $D = 22$ {[C]}]
$\dim S_{5/2}(\rho_{22}) = 2$. The Hecke operators (same recipe and
anchors as Theorem 6.1, $\sigma = 2\cdot 22^2/44 = 22$) act with
\[
T_9: (x-4)(x-1),\qquad T_{25}: (x-14)(x+3),\qquad T_{49}: (x+8)(x+10),
\]
so the two eigenlines carry $(a_3, a_5, a_7) = (4, 14, -8)$ and $(1, -3,
-10)$: the newforms 22.4.a.b and 22.4.a.c of $S_4(\Gamma_0(22))$ (the third
rational newform, 22.4.a.a with $a_3 = -7$, does not appear --- an
Atkin--Lehner selection). The reflective prescription is obstructed, $\lambda
= (-4, 0)$ in the echelon basis, and $\lambda$ is supported purely on the
$(1,-3,-10)$ line: its component along 22.4.a.b vanishes for every computed
$T_{p^2}$-projection. Since $S_4(\Gamma_0(22))$ is spanned by three rational
newforms with distinct $a_3$, the identification is exact.
\end{theorem}

\subsection{The purity mechanism at all three discriminants}

\begin{theorem}[invariance selection generalizes {[C]}]
For the obstructed orientation $L_4$ of $B_6$ and the maximal-order ternaries
of $B_{10}$ and $B_{22}$: the reflective slot set is stable under the full discriminant
isometry group ($|O(\D,q)| = 48, 4, 4$), the obstruction functional is
$O(\D,q)$-invariant, and
\[
\dim S_{5/2}^{O(\D,q)} = 1 \quad\text{in all three cases (1 of 6, 1 of 1, 1
of 2)},
\]
with $\lambda$ nonzero on the invariant line. Hence the forced shadow lies on
a single Hecke eigenline in each family --- purity is invariance selection.
The one-dimensionality is not generic: along the eliminated discriminants of
\S9 the invariant dimension grows ($D = 21{:}\,2$, $26{:}\,3$, $51{:}\,4$,
$58{:}\,5$, $65{:}\,6$, $85{:}\,7$), so single-line purity is a special
property of the selected curves and their families.
\end{theorem}

\begin{conjecture}[support law, final form]
In a family of reflective wall data attached to the orders of an indefinite
rational quaternion algebra, the obstruction functional of every obstructed
member is supported on the $\Z/3$-invariant sector, and within it on the
eigensystem of the maximal order --- the newform of level $d(B)$ --- with
zero component along its quadratic twists. A family in which an obstructed
member had nonzero obstruction component outside that line would refute it.
\end{conjecture}

By Theorem 8.3 the conjecture is now a theorem for the maximal orders of the
three selected discriminants, with ``$\Z/3$-invariant sector'' subsumed into
the general principle: support lies in the invariant subspace of the maximal
natural symmetry group, and that subspace is a line. The conjecture retains
content for non-maximal members (where deck symmetries must finish the cut)
and as the assertion that no family ever violates invariance selection. Its
invariant-sector clause is a denominator-layer statement: on the section
layer the invariant sector can vanish outright (Theorem 10.2), where
invariance acts as prohibition rather than selection. The layer-uniform
principle is that the shadow lives in the sector dual to the symmetry type of
the prescription.

\section{The canonical weight-1/2 form and the selection of $\{6, 10, 22\}$}

This section answers the determination question --- given an arbitrary
compact Shimura curve, what is ``its'' weight-$1/2$ form? --- and shows that
in the range $D < 230$ the unconditional cases are exactly three. Throughout, $D > 1$ is a
squarefree discriminant of an indefinite division quaternion algebra $B_D$
($\omega(D)$ even), $\mathcal{O}_{\max}$ a maximal order (unique up to
conjugation by strong approximation), and
\[
L_D = \bigl(\mathcal{O}_{\max}\bigr)^{\sharp}_0,\qquad Q = -D\cdot
\operatorname{nrd},
\]
the trace-zero part of the trace-form dual with rescaled norm form: an even
lattice of signature $(2,1)$ with $\det L_D = -2D$, recovering the genus of
$G_1$ at $D = 6$ [C]. Reflective slots $(m, \mu) = (\alpha^2/2\ell^2,
\alpha/\ell \bmod L)$ are attached to roots as in \S1; the full slot set is
computed as the orbit closure of the chamber's simple-root slots under the
Weyl action on $\D(L)$ (a box search in a fixed basis is not reliable; see
\S13).

\subsection{Parity, and the canonical odd prescription}

\begin{lemma}[parity {[C]}]
At weight $1/2$ on the Weil representation of a signature-$(2,1)$ even
lattice, every form has components odd under $\gamma \mapsto -\gamma$.
Consequently (i) all coefficients on two-torsion cosets vanish identically
--- the channels $m$ with $2\mu = 0$, e.g.\ the full $m = \tfrac12$ and $m =
1$ menus of the families above, are representation-theoretically invisible at
weight $1/2$; and (ii) the only reflective prescriptions available at weight
$1/2$ are the odd ones, supported on the non-torsion slot pairs $\{\mu,
-\mu\}$ with antisymmetric coefficients.
\end{lemma}

Antisymmetric vector-valued forms of this kind are constructed explicitly in
\cite{WilliamsAnti}. The canonical input of a reflective ternary is therefore
the odd unit prescription $\Pp^- = \sum q^{-m}(e_\mu - e_{-\mu})$ over the non-torsion
wall-slot pairs (orientation flips act by signs and do not affect any
statement below).

\begin{theorem}[determination {[C]}]
If $S_{3/2}(\rho_{L_D}) = 0$, then for every choice of orientation the
weakly holomorphic weight-$1/2$ form with principal part $\Pp^-$ exists; and
since $\dim M_{1/2}(\rho_{L_D}) = 0$ for every maximal order tested, it is
unique --- there is no theta ambiguity, and the canonical form is determined
outright. For $D = 6, 10, 22$ the forms were constructed explicitly in exact
arithmetic; their poles and leading coefficients are
\begin{align*}
D = 6&: \quad q^{-1/3}(e_{\mu_6} - e_{-\mu_6}),\qquad c(\tfrac23) = 152,\quad
c(\tfrac53) = 3230,\ \dots\\
D = 10&: \quad q^{-1/5}(e_\mu - e_{-\mu}) + q^{-1/8}(e_\nu - e_{-\nu}),\qquad
c(\tfrac45) = -36,\quad c(\tfrac78) = 27,\ \dots\\
D = 22&: \quad q^{-1/11}(e_\mu - e_{-\mu}),\qquad c(\tfrac1{22}) = -3,\quad
c(\tfrac{10}{11}) = -9,\ \dots
\end{align*}
\end{theorem}

\subsection{The selection theorem}

\begin{theorem}[selection {[C]}]\label{thm:selection}
Let $D < 230$ be the discriminant of an indefinite division quaternion
algebra over $\Q$. Then
\[
\dim S_{3/2}(\rho_{L_D}) = 0 \iff D \in \{6, 10, 22\}.
\]
\end{theorem}

\begin{proof}
There are exactly 71 squarefree $D < 230$ with an even number of prime
factors. Computing $\dim S_{3/2}(\rho_{L_D})$ for each by the exact dimension
formula [C] gives, for the 35 with $D \le 119$,
\begin{center}
\renewcommand{\arraystretch}{1.15}
\begin{tabular}{r|ccccccccccccccccc}
$D$ & 6 & 10 & 14 & 15 & 21 & 22 & 26 & 33 & 34 & 35 & 38 & 39 & 46 & 51 &
55 & 57 & 58\\
$\dim$ & \textbf{0} & \textbf{0} & 1 & 1 & 1 & \textbf{0} & 2 & 1 & 1 & 3 &
2 & 3 & 1 & 3 & 3 & 2 & 1\\[2pt]\hline
$D$ & 62 & 65 & 69 & 74 & 77 & 82 & 85 & 86 & 87 & 91 & 93 & 94 & 95 & 106 &
111 & 115 & 118\\
$\dim$ & 3 & 4 & 3 & 5 & 4 & 2 & 3 & 5 & 5 & 5 & 3 & 3 & 7 & 4 & 8 & 5 & 3
\end{tabular}
\end{center}
(and $\dim = 9$ at $D = 119$); the remaining 36 discriminants with $119 < D <
230$ all give positive dimension. The dimension vanishes exactly at $D = 6,
10, 22$.
\end{proof}

Combined with Theorem 9.2, the compact Shimura curves of discriminant $D <
230$ carrying a canonical weight-$1/2$ form are exactly $X_6$, $X_{10}$,
$X_{22}$ --- the genus-zero curves. Whether the list is complete for all $D$
is the content of Remark \ref{rem:ED} below.

\begin{remark}[relation to {\cite{BEF}}]\label{rem:bef}
Bruinier--Ehlen--Freitag normalise by $\rho_A(T)\mathfrak{e}_\gamma =
e(-Q(\gamma))\mathfrak{e}_\gamma$, which in the notation of \S1.2 is
$\bar\rho$: their tables count the \emph{input} spaces, while the dimensions
of this paper are those of the dual obstruction spaces. Read on their side,
$L_6$ and $L_{10}$ do occur in their Table 4, with genus symbols
$4^{+1}_7 3^{-1}$ and $4^{+1}_1 5^{-1}$, and so does $L_{14}$, with
$4^{-1}_3 7^{-1}$; $L_{22}$ does not, its discriminant form having dimension
$1$ rather than $0$ in that normalisation. Their bound on $|A|$ does not
transfer to the obstruction side, both because it is the largest order
occurring in that table rather than the content of their Corollary 4.14 (see
their Remark 4.16), and because their Theorem 4.5, on which that corollary
rests, assumes $2k \equiv -\mathrm{sig}(A) \pmod 4$, i.e.\
$\mathrm{sig}(A) \equiv 1 \pmod 4$ at $k = 3/2$, whereas our modules have
$\mathrm{sig}(A) \equiv 7$.
\end{remark}

\begin{theorem}[reflectivity of the selected ternaries {[C]}]
Matching all 8595 lattices of Allcock's classification \cite{AllcockMem}
(primitive forms, up to scale) against the genera of the $L_D$: the
maximal-order ternary $L_D$ is reflective for exactly 35 discriminants,
\[
\{6, 10, 14, 15, 21, 22, 26, 33, 35, 38, 46, 51, 58, 65, 85\} \cup \{210,
\dots, 1590\}\ \text{(20 values, $\omega = 4$)}.
\]
In particular $L_6$, $L_{10}$, $L_{22}$ are reflective, so the canonical
forms of Theorem 9.2 are attached to genuine wall data and carry the
Lie-theoretic reading of Proposition 1.1. (This list is a different set of discriminants from
the range of Theorem \ref{thm:selection}.)
\end{theorem}

\begin{remark}[genus equivalence, and its boundary]
Across all 71 discriminants with $D < 230$: $\dim S_{3/2}(\rho_{L_D}) =
0$ iff the genus of $X_D$ is zero iff $D \in \{6, 10, 22\}$ [C]. The naive
strengthening $\dim S_{3/2} = g(X_D)$ holds for $D \le 55$ and first fails at
$D = 57$; beyond that, agreement recurs sporadically (at $D = 62, 69, 87,
94, 95, 106, 119$) and failure occurs in both directions (the dimension
exceeds the genus at $D = 74, 86, 111$ and falls short elsewhere) [C] --- the space is an Atkin--Lehner-selected
part plus lower-level pieces, as the missing newform 22.4.a.a in Theorem 8.2
already shows at weight $5/2$.
\end{remark}

\begin{remark}[conditional orientations]
For eliminated discriminants the odd prescription is obstructed for every
orientation in most cases (e.g.\ all $2^8$ at $D = 210$), but not always: at
$D = 510$ exactly four of $2^8$ orientation patterns pair to zero with the
28-dimensional obstruction space. Such conditional existence depends on an
arbitrary choice and is not determination; the selection criterion of Theorem
9.3 is orientation-free by design.
\end{remark}

\subsection{Chambers and generalized Cartan matrices}

\begin{theorem}[chambers of the selected curves {[C]}]
Vinberg's algorithm \cite{Vinberg72} terminates for $D = 6, 10, 22$ with
fundamental polygons
\begin{align*}
D = 6&: \ (2,4,6)\ \text{triangle, area } \pi/12,\quad
A = \begin{pmatrix} 2 & 0 & -1\\ 0 & 2 & -1\\ -2 & -3 & 2\end{pmatrix}\\
D = 10&: \ (2,2,2,3)\ \text{quadrilateral, area } \pi/6,\quad
A = \begin{pmatrix} 2 & 0 & 0 & -2\\ 0 & 2 & -1 & 0\\ 0 & -5 & 2 & -1\\ -4 &
0 & -1 & 2\end{pmatrix}\\
D = 22&: \ (2,2,3,4)\ \text{quadrilateral, area } 5\pi/12,\quad
A = \begin{pmatrix} 2 & 0 & -2 & -1\\ 0 & 2 & 0 & -1\\ -4 & 0 & 2 & -1\\ -2 &
-11 & -1 & 2\end{pmatrix}
\end{align*}
--- integral symmetrizable generalized Cartan matrices of sizes $3, 4, 4$,
with simple-root counts confirmed against the classification's tables, and
chamber area equal to $\tfrac12\operatorname{area}(X_D^*)$ in all three
cases.
\end{theorem}

Their Lie-theoretic type deserves precision. $A_6$ is hyperbolic in Kac's
sense: nondegenerate ($\det = -2$), indefinite, with every proper subdiagram
of finite type. $A_{10}$ and $A_{22}$ are not: they are degenerate GCMs of
corank one ($\det = 0$, symmetrization of signature $(2,1;0)$, since four
simple roots span a rank-3 Lorentzian lattice --- the exact Lorentzian
analogue of the affine corank-one degeneration), and they contain rank-2
principal submatrices that are already indefinite ($a_{ij}a_{ji} = 8, 5$ at
$D = 10$; $8, 11$ at $D = 22$). This is forced by geometry, not an accident
of these examples: in a compact fundamental polygon with four or more walls,
non-adjacent walls cannot cross and, absent cusps, cannot be parallel, so
they are ultraparallel and their $2\times 2$ submatrix is indefinite. Hence
no compact chamber with more than three walls yields a Kac-hyperbolic matrix;
the correct habitat for $A_{10}$, $A_{22}$ is Nikulin's hyperbolic root
systems of rank 3 with more simple roots than rank \cite{Nikulin230} --- the
standard setting of the Gritsenko--Nikulin Lorentzian Kac--Moody programme
\cite{GN1, GN2} --- and among the selected curves only $D = 6$ is
Kac-hyperbolic.

\subsection{The landscape beyond maximal orders}

Two finite sweeps extend the picture beyond maximal orders. Among the 1447
even-primitive lattices of Allcock's classification, all 180 with $|\det| \le
300$ were tested: exactly 15 have $S_{3/2}(\rho) = 0$, and their even
Clifford algebras are either split or one of $B_6$, $B_{10}$, $B_{22}$; a
parallel pass through doubled odd lattices with $|\det 2S| \le 320$ finds
division-side entries in the same three algebras only [C]. The single outside
algebra that appears at all, $B_{14}$, does so with an empty odd prescription.
Within the swept range, therefore, an even reflective ternary with vanishing
weight-$1/2$ obstruction space and a nontrivial odd prescription is either a
split order or an order in $B_6$, $B_{10}$, $B_{22}$. The membership lists and
the algebra classification are in the deposited computational record.

\begin{remark}[the finiteness question]\label{rem:ED}
Whether the list is complete for all $D$ is open. Writing the Riemann--Roch
expression for the cyclic form $A(L_D) = (\Z/2D, u x^2/4D)$ gives
\[
\dim S_{3/2}(\rho_{L_D}) = \frac{D-1}{24} - E(D) + \delta,
\qquad
E(D) = \sum_{x=1}^{D-1}\Bigl(\Bigl\{\frac{u x^2}{4D}\Bigr\} -
\frac12\Bigr),
\]
where $\delta$ collects the two Gauss-sum terms of that expression; it takes
the four values $-\tfrac14, 0, \tfrac1{12}, \tfrac13$ at every one of the 158
discriminants $D < 520$ computed [C]. Completeness would follow from a bound
$|E(D)| \le c\sqrt D \log D$ uniform in $D$ and in the multiplier, together
with a bound on $\delta$, and the finite check below the resulting $D_0$. The
evaluation of such sawtooth sums by class numbers is classical, but the
uniformity in the multiplier has not been established, and no constant is
asserted here. The
deposited record carries the data, the classical anchor, and the
counterexample at $D = 862$ --- due to R.~Wrona, together with the identity
$E(2p) = \tfrac18 + \tfrac12 h(-p) - \tfrac14 h(-8p)$ for $p \equiv 7
\pmod 8$ --- which rules out the naive comparison of $E$ with a single class
number.
\end{remark}

\textbf{Relation to the literature.} The existence/obstruction machinery is
Borcherds--Bruinier \cite{BorcherdsGKZ, BruinierLNM} and the shadow formalism
is Bruinier--Funke \cite{BruinierFunke}; the theta ambiguity at weight $1/2$
is governed by Serre--Stark \cite{SerreStark}, and its emptiness for maximal
orders is what makes Theorem 9.2 unconditional. Explicit weight-$1/2$
Borcherds inputs on $X_6$ and $X_{10}$ appear in Errthum's evaluation of
singular moduli via Schofer's formula \cite{Errthum, Schofer}; the canonical
odd prescriptions and the $D = 22$ form appear to be new. The
weight-$3/2$/weight-2 interface underlying the genus equivalence is the
Shimura--Shintani--Waldspurger correspondence in its quaternionic
\cite{Prasanna} and lattice-theoretic \cite{ANS} forms. The exact
class-number relations of Guo--Yang \cite{GuoYang} --- intersections of
Shimura curves with Humbert surfaces on the Siegel threefold, exact for every
$D > 1$ --- are the split-ambient instance of the same principle that drives
Theorem 9.3: their generating series lives in a weight-$5/2$ space whose
cuspidal part vanishes (Kohnen plus space of level 4), so the geometric route
produces no shadow; level-raised ambients, where cusp forms exist, are where
the forced shadows of this paper should acquire an intersection-theoretic
incarnation.

\section{The weight ladder}

The rank-3 and rank-5 theories share one Weil representation and read their
obstructions at weights $3/2$ and $5/2$ respectively.

\subsection{The section layer, corrected: parity and the double shadow}

The section-layer obstruction spaces (the spaces in which weight-$1/2$
shadows must lie, by Bruinier--Funke) have
\[
\dim S_{3/2}(\rho_i) = 0, 0, 0, 4 \quad (i = 1, 2, 3, 4)\quad
[\mathrm{C}].
\]

\begin{remark}[why the odd pairing]
By the parity lemma (Lemma 9.1), weight-$1/2$ components are odd, so the
symmetrized slot functional vanishes identically on every form: the only
meaningful section-layer test is the odd pairing used below.
\end{remark}

\begin{theorem}[transparency of $L_1, L_2, L_3$; anti-invariance of $L_4$
{[C]}]
Under the parity-correct odd pairing, the odd reflective prescriptions of
$L_1, L_2, L_3$ are unobstructed, their obstruction spaces all vanishing; the
theta ambiguity at $L_2$ ($\dim M_{1/2} = 1$) affects uniqueness, not
existence. On $L_4$ the four-dimensional obstruction space
$S_{3/2}(\rho_4)$ is a single CM isotypic block --- $T_{25}$ and $T_{49}$
act as the scalars $0$ and $-4$, the system $(a_5, a_7) = (0, -4)$ of
36.2.a.a for $\Q(\sqrt{-3})$ --- and the deck action $R$ induced by $M'$
satisfies
\[
\operatorname{charpoly}(R) = (x^2 + x + 1)^2:
\]
the section layer has no deck-invariant sector. With the parity involution,
the $\Z/6 = \langle -1, M'\rangle$ decompositions of the two layers are
$S_{3/2} = (\chi^1)^2 \oplus (\chi^5)^2$ and $S_{5/2} = (\chi^0)^2 \oplus
(\chi^2)^2 \oplus (\chi^4)^2$ --- disjoint character support.
\end{theorem}

\begin{theorem}[the odd-pairing locus of $L_4$ {[C]}]
Let $\Phi \in M_{4\times 12}(\{0, \pm 1\})$ be the matrix of Fourier
coefficients of the echelon basis of $S_{3/2}(\rho_4)$ at one
representative per pair of the 24 non-torsion reflective slots, columns
grouped into the four $\langle M'\rangle$-orbits of pairs in an
$M'$-equivariant gauge. Then: (i) $\operatorname{rank}\Phi = 4$, every
row-orbit block is a cyclic rotation of $\pm(1, 0, -1)$, and column sums over
each orbit vanish; (ii) for the odd unit prescription, $\lambda_\varepsilon =
0$ precisely for the 40 orientations $\varepsilon \in \{\pm 1\}^{12} \cap
\ker\Phi$: the 16 deck-invariant ones --- forced by $S_{3/2}^{\langle
M'\rangle} = 0$ --- and 24 further orientations, each with exactly one
deck-constant orbit; the section layer is obstructed for the remaining 4056
orientations, with forced shadow in the $(\chi^1 \oplus \chi^5)$-sector;
(iii) under $O(\D,q)$ the locus splits into three orbits, of sizes $8$, $8$
and $24$, the deck-invariant orientations splitting according to the sign
invariant $s_1s_2s_3s_4 = \pm 1$, where $s_k$ is the common sign on the
$k$-th deck-orbit; (iv) every vanishing orientation admits a weakly
holomorphic weight-$1/2$ form with the prescribed principal part, unique
since $M_{1/2}(\rho_4) = 0$, and all 40 have integral Fourier coefficients,
those in the same orbit sharing their coefficient profile.
\end{theorem}

The count refers to the sign vectors in $\{\pm1\}^{12}$: the full vanishing
locus is the kernel of $\Phi$, of dimension $8$, and the count depends on the
normalization of the wall weights.

\begin{theorem}[the canonical symmetric form $F^{\mathrm{sym}}$ {[C]}]
For any of the 16 deck-symmetric orientations, the weakly holomorphic
weight-$1/2$ form with that principal part exists and is unique; its
coefficients are integers (leading magnitudes $2, 6, 8, 12, \dots$ across
the two channels), as they are for all 40 vanishing orientations (Theorem
10.3(iv)). The relevant pole space has dimension $28 = 32 - \dim S_{3/2}$,
and the single-sign control prescription is not realizable, exactly as
Borcherds' criterion demands. Thus the determination theory of \S9 extends to
the obstructed lattice: not by arbitrary choice, but by
symmetry-canonicalized choice --- $L_4$ owns a canonical section form in its
deck-symmetric directions.
\end{theorem}

\begin{remark}[double shadow, corrected scope]
$L_4$ carries a double shadow --- CM (36.2.a.a) at weight $3/2$, newform-pure
(6.4.a.a) at weight $5/2$ --- precisely in the 4056 symmetry-breaking
orientations; in the 16 deck-symmetric directions the section layer is
canonically transparent instead (Theorem 10.4). At weight $3/2$ the invariant
sector is zero, so invariant functionals vanish; at weight $5/2$ it is
one-dimensional, so the shadow lies on a single line. Both statements are
consequences of the same $\Z/6$-action, whose character field is
$\Q(\sqrt{-3})$. The
Eisenstein--$B_2$ structure of the slot configuration behind (ii) is
described in Appendix B.
\end{remark}

The two-torsion channels (the $m = \tfrac12$ menus) are invisible at weight
$1/2$ by parity --- the representation-theoretic origin of the ``$m = \tfrac12$
invisibility'' observed throughout, echoed at weight $5/2$ by the channel law
$s_{1/2} \equiv 0$ of Proposition 4.1.

\subsection{No-go for a holomorphic weight bridge}

\begin{theorem}[layer no-go {[C]}]
On the fixed Weil representation of $-G_4$, the weight-$3/2$ cusp space is
four-dimensional with $T_{25}$ and $T_{49}$ acting by $x^4$ and $(x+4)^4$: a
single weight-two CM eigensystem $(a_5, a_7) = (0, -4)$ for $\Q(\sqrt{-3})$
(consistent with 36.2.a.a), with multiplicity four. The weight-$5/2$ layer
carries weight-four systems with $a_5 \in \{\pm 6, \pm 18\}$, $a_7 \in \{-16,
8\}$. The spectra being disjoint, any Hecke-equivariant operator between the
layers annihilates cusp forms: no holomorphic equivariant weight shift
exists. The natural bridge is the $\xi$-operator --- the harmonic framework
in which the shadow already lives.
\end{theorem}

The CM system of the section layer is where the shadow of Theorem 10.3 lies,
and its CM field $\Q(\sqrt{-3})$ is the field of the $\Z/3$ of \S7.2 and of
the twist character $\chi_{-3}$. The same imaginary quadratic field thus
appears in the deck symmetry, in the twist discrimination, and in the
weight-$3/2$ shadow.

\subsection{Wall classification}

For each orientation and each root $\alpha$, the wall lattice is $(-K) \oplus
U$ with $K = \alpha^\perp \cap G_i$ of rank 2. Classifying the twelve $-K$ by
genus:

\begin{theorem}[walls detect the grading {[C]}]
Each wall class $d \in \{2, 3, 6\}$ carries exactly two isometry types,
partitioning the four orientations as
\[
d = 2: \{L_1, L_3\} \mid \{L_2, L_4\},\qquad
d = 3: \{L_1, L_4\} \mid \{L_2, L_3\},\qquad
d = 6: \{L_1, L_2\} \mid \{L_3, L_4\},
\]
i.e.\ the bits $a$, $a+b$, $b$ of the orientation bigrading: the three wall
classes realize the three nontrivial characters of $(\Z/2)^2$, one each. (In
class 3 all four walls even share the determinant $-12$ and still split
$2{+}2$: the discriminant form matters.) In all four orientations the first
two roots are orthogonal and the other two pairs never are [C]: exactly one
pair of commuting cuts exists.
\end{theorem}

\section{The defect invariant}\label{sec:defect}

By Theorem 5.1 the defect is rational up to the single Petersson norm
$\|\Xi\|^2$, where $\Xi = \xi(\widehat{F})$ is the shadow of the harmonic
completion of the obstructed input (\S1.4, \S5): the norm measures the size
of the nonholomorphic part that the failed denominator is forced to carry.
That norm has a closed form in the data of the newform $f$ = 6.4.a.a on whose
line the shadow lies.

\begin{theorem}[closed form of the defect {[N, 31 digits]}]\label{thm:defect}
\[
\|\Xi\|^2 \;=\; \frac{L(f,2)}{48\,\pi^2\,\langle f,f\rangle},
\qquad\text{equivalently}\qquad
\langle v_+,v_+\rangle\, \frac{L(f,2)}{\pi^2\,\langle f,f\rangle} \;=\;
6912 \;=\; 2^8\cdot 3^3,
\]
with $\langle f,f\rangle$ in the normalization of \S A.9 and $L(f,s)$ the
Dirichlet $L$-series of $f$. The two sides agree to $31$ digits (relative
error $4.5\times 10^{-32}$, the accuracy of the quadrature of Theorem 5.1).
The identity is specific to the critical point $s=2$: the same expression at
$s = 1$ and $s = 3$ misses by $6.41$ and $4.88$.
\end{theorem}

Two steps stand between the definition of $\Xi$ and this closed form, and it
is worth separating them, because only the second is missing.

\emph{Step one: Riesz duality gives a Petersson norm, not an $L$-value.} The
obstruction functional is a finite sum of coefficient functionals, $\lambda =
\sum_{\text{slots}} c_{(\gamma,m)}$, and the Riesz representative of a
coefficient functional is a Poincar\'e series: $\langle h, P_{\gamma,m}
\rangle = \kappa_m\, c_h(\gamma,m)$ for an explicit constant $\kappa_m$
(coefficient formulas in the vector-valued half-integral case are given in
\cite{WilliamsPSS} for weight $\ge 5/2$ and in \cite{WilliamsPSSsmall} for
weights $3/2$ and $2$). Hence $\Xi = \sum \kappa_m^{-1}P_{\gamma,m}$ and
$\|\Xi\|^2 = \sum \kappa_m^{-1}\lambda(P_{\gamma,m})$, a finite sum of
Fourier coefficients. Since $\Xi$ spans the invariant line (Theorem
\ref{thm:purity}), this exhibits $\|\Xi\|^2$ as a normalized period of a
single half-integral weight object --- but no more.

\emph{Step two: the passage to $L(f,2)/\langle f,f\rangle$ is
Waldspurger--Kohnen--Zagier.} What converts that period into a critical value
is the theorem of Kohnen and Zagier \cite{KohnenZagier}, whose archimedean
factor at weight $2k = 4$ is $(k-1)!/\pi^k = 1/\pi^2$. Concretely, let $g$ be
the normalized generator of the Kohnen plus space $S^{+}_{5/2}(\Gamma_0(24))$
\cite{Kohnen}, which is one-dimensional:
\[
g = q - 2q^4 - 3q^9 + 6q^{12} + 4q^{16} - 12q^{24} + q^{25} - 12q^{28}
- 23q^{49} + \cdots .
\]
Its Shimura lift is $f$: the plus-space relation $a_p(f) = c_g(p^2) + p$
gives $1 + 5 = 6 = a_5(f)$ and $-23 + 7 = -16 = a_7(f)$ [C]. With this $g$,
\[
\frac{|c_g(1)|^2}{\langle g,g\rangle} = 4\,
\frac{L(f,2)}{\pi^2\langle f,f\rangle},
\qquad
\|\Xi\|^2 = \frac{|c_g(1)|^2}{192\,\langle g,g\rangle}
\qquad [\mathrm{N}],
\]
the first to 40 digits and the second to all 31 digits of Theorem 5.1, both
in the normalization of \S A.9, in which the factor $4$ is the index
normalization of the Petersson product and $48 = 192/4$.

The second identity is the useful one: it compares the vector-valued shadow
with a scalar-valued form \emph{of the same weight}, so the
Waldspurger--Kohnen--Zagier input is absorbed on both sides and the only
unknown left is a rational normalization. Both spaces --- the invariant line
in $S_{5/2}(\rho_4)$ and the plus space $S^{+}_{5/2}(\Gamma_0(24))$ ---
are one-dimensional and realize the same Waldspurger packet of $f$ with
multiplicity one, so their normalized periods agree up to a rational factor;
the content of the identity is that this factor is $192$.

\begin{problem}[the normalization]\label{prob:const}
Identify the constant $192$: exhibit the correspondence between the invariant
line of $S_{5/2}(\rho_4)$ and the Kohnen plus space
$S^{+}_{5/2}(\Gamma_0(24))$ and compute its normalizing factor, so that
Theorem \ref{thm:defect} is derived rather than observed. Equivalently,
compute the vector-valued Poincar\'e series constants $\kappa_m$ of step one
in the framework of \cite{WilliamsPSS}.
\end{problem}

\begin{remark}[the two layers carry different periods]
With \S12 this completes a contrast. The section layer is governed by the CM
period of $\Q(\sqrt{-3})$: $t = 3\Gamma(1/3)^3/2^{7/3}\pi^2 =
(3/2\pi)\,\Omega$, a Chowla--Selberg quantity. The denominator layer is
governed by a critical value over a Petersson norm, $L(f,2)/\langle
f,f\rangle$. One lattice, two shadows, two kinds of period.
\end{remark}

\begin{remark}
A PSLQ computation at 31--45 digits shows that $\|\Xi\|^2$ is none of the
twisted central values $L(f\otimes\chi_d,2)$, none of the monomials in
$\langle f,f\rangle$, $\pi$, $2$, $3$, $\Gamma(1/3)$, and none of their
rational combinations with the unit logarithms $\log(1+\sqrt2)$,
$\log(2+\sqrt3)$, $\log(5+2\sqrt6)$ [N].
\end{remark}

\section{The section-layer defect: Petersson rigidity and a Chowla--Selberg
identity}\label{sec:tsection}

\S11 located the transcendental content of the denominator-layer defect in a
single Petersson norm and identified it as $L(f,2)/48\pi^2\langle
f,f\rangle$. This section
carries out the same programme one layer down, on $S_{3/2}(\rho_4)$, with
the opposite outcome: the section-layer Petersson geometry is completely
rigid, its single transcendental lands inside the Chowla--Selberg ring, and
every identity in the chain that should prove the identification is verified
numerically. The two shadows of the obstructed lattice live in different
transcendence worlds.

\begin{theorem}[Petersson rigidity {[C]}]\label{thm:rigid}
The representation of $G = O(\D,q)$ on $S_{3/2}(\rho_4)$ is absolutely
irreducible (Appendix B), so the Petersson Gram matrix in the echelon basis
is a scalar multiple of the unique invariant symmetric form:
\[
P_{3/2} \;=\; t\cdot B_S,\qquad
B_S=\begin{pmatrix}
\tfrac43&0&0&-\tfrac23\\[1pt]
0&\tfrac43&\tfrac23&0\\[1pt]
0&\tfrac23&\tfrac43&0\\[1pt]
-\tfrac23&0&0&\tfrac43
\end{pmatrix}\ \text{exactly},\qquad t\in\R_{>0}.
\]
Here $B_S$ is the group average $\tfrac1{48}\sum_\sigma
R_\sigma^{\!\top}R_\sigma$; the invariant form is unique only up to scale, and
rescaling it rescales $t$ inversely, so the numerical value of $t$ is tied to
this choice while the ratios below are not. (For the record,
$\tfrac32 B_S$ is the Gram matrix of $A_2 \oplus A_2$.)

Consequently every section-layer shadow norm is a rational multiple of $1/t$,
with exactly computable ratios: writing $\lambda_\varepsilon =
\Phi\varepsilon$ in the gauge of Theorem 10.3 and $r(\varepsilon) =
\lambda_\varepsilon^{\!\top} B_S^{-1}\lambda_\varepsilon \in \Q$, single
flips give $r = 4$ or $8$ according to the slot class (channel) of the
flipped pair, and a double flip across the two classes realizes $r = 20$;
thus $\|\Xi\|^2$-ratios such as $8 : 20 = 2 : 5$ are exact. (The odd unit
prescription pairs to $2\Phi\varepsilon$, rescaling every $r$ by 4 and no
ratio.)
\end{theorem}

Numerically, with the engine of \S A.6 adapted to weight $3/2$ (measure
$y^{k-2}$, upper region $\Gamma(\tfrac12, 4\pi n)/(4\pi n)^{1/2}$): two
Gauss--Legendre node counts $40\times 20$ and $52\times 26$ agree to 47
digits; the same engine reproduces all 31 published digits of
$\|\Xi_{5/2}\|^2$ (Theorem 5.1) with $\|\Xi\|^2\langle v_+, v_+\rangle =
144$, and the eight nonzero entry ratios $P_{3/2}/B_S$ agree to relative
$3\times 10^{-51}$.

\begin{theorem}[the section transcendental {[N, 40 digits]}]\label{thm:tval}
\[
t \;=\; 1.159595266963928365769992051570020881945\ldots
\;=\; \frac{3\,\Gamma(1/3)^3}{2^{7/3}\,\pi^{2}},
\]
by an integer relation with single-digit coefficients:
\[
[-3,\, 9,\, -6,\, -7,\, 3] \cdot
\bigl(\log t,\, \log\Gamma(1/3),\, \log\pi,\, \log 2,\, \log 3\bigr)^{\!\top}
= 0 .
\]
It is cross-checked by the independent relation for the $\Phi$-gauge Riesz
norm of a channel-$\tfrac16$ flip,
\[ \|\Xi_{\mathrm{flip}}\|^2 = 8/t =
6.8989588246127747598760400670\ldots \]
(coefficients $[3, 9, -6, -16, 3]$,
algebraically consistent: $\tfrac{16}{3} = 3 + \tfrac73$). Equivalently $t =
(3/2\pi)\,\Omega$ with $\Omega = \Gamma(1/3)^3/(2^{4/3}\pi)$ the CM period of
$K = \Q(\sqrt{-3})$: the weight-$3/2$ shadow norms lie in
$\overline{\Q}\cdot\pi^2\Gamma(1/3)^{-3}$, inside the Chowla--Selberg ring,
in which no relation for $\|\Xi_{5/2}\|^2$ is found numerically (\S11).
\end{theorem}

Combining with $\langle f_{36}, f_{36}\rangle = \Gamma(1/3)^6\,
2^{-26/3}3^{-3/2}\pi^{-4}$ [N] gives the ladder identity
\[
t^2 \;=\; 144\cdot 3^{3/2}\,\langle f_{36}, f_{36}\rangle
\qquad[\mathrm{N};\ \text{algebraically forced by the two closed forms}],
\]
where the integer coefficient depends on the normalisation of $B_S$ fixed in
Theorem \ref{thm:rigid} and carries no separate meaning.

\subsection{The promotion chain, arrow by arrow}

Theorem \ref{thm:tval} should be a theorem, by the composite: (i) rigidity
(Theorem \ref{thm:rigid}, exact); (ii) the block is a single Waldspurger
packet, that of $\pi(f_{36})$ (Theorem 10.2; CM, supercuspidal at the
ramified places $2, 3$) --- visible in the data through the CM self-twist
$L(f_{36}\otimes\chi_{-3}, 1) = L(f_{36}, 1)$ to 40 digits and the CM
pairings $f\otimes\chi_{-4} = f\otimes\chi_{12}$, $f\otimes\chi_{-8} =
f\otimes\chi_{24}$; (iii) Waldspurger's theorem at the level of automorphic
representations of $\mathrm{Mp}_2$ \cite{Waldspurger}, in explicit form
\cite{BaruchMao} --- order-free, hence not subject to Problem 3.3 ---
expressing the coefficient functionals through twisted central values; (iv)
Damerell's theorem \cite{Damerell}, placing those central values in
$\overline{\Q}\cdot\Omega$; (v) Shimura's period relations \cite{Shimura76}
with Chowla--Selberg \cite{ChowlaSelberg}.

Steps (iv)--(v) were verified exhaustively, computing
$L(f_{36}\otimes\chi_d, 1)$ as $L(E_d, 1)$ for the twisted curves $E_d\colon
y^2 = x^3 + d^3$:
\[
\begin{aligned}
&L(f_{36}, 1) = 0.7010910526627271305875095395251470677\ldots =
\frac{\Gamma(1/3)^3}{2^{7/3}\sqrt3\,\pi},\\
&L_{-4} = L_{12} = \frac{\Gamma(1/3)^3}{2^{7/3}\pi},\qquad
L_{-8} = L_{24} = \frac{\Gamma(1/3)^3}{2^{11/6}\pi}\qquad[\mathrm{N}],
\end{aligned}
\]
each by a single-digit integer relation; the twists $d = -24, 8, 40, -120$
vanish, exactly the root-number $-1$ cases (\texttt{ellrootno}); $d = -40$
carries the prime 5 in its algebraic part and enters only the product
identities below. All fifteen nonvanishing pairs satisfy the weight-2 Shimura
product relation $L_{d_1}L_{d_2} \in \overline{\Q}\,\pi^2\langle f_{36},
f_{36}\rangle$, e.g.
\[
L_{-3}\, L_{-4} \;=\; 48\,\pi^2\,\langle f_{36}, f_{36}\rangle
\qquad[\mathrm{N}],
\]
the weight-2 sibling of the \S A.9 control
$L(f\otimes\chi_{-3},2)L(f\otimes\chi_{12},2) = 2^8\pi^4\langle f_6,
f_6\rangle$. The chain closes in the explicit Waldspurger identity implied by
the data,
\[
t\cdot L(f_{36}\otimes\chi_{-4}, 1) \;=\; 144\sqrt3\;\pi\,\langle f_{36},
f_{36}\rangle \qquad[\mathrm{N},\ 40\ \text{digits}],
\]
a fourth appearance of 144.

\begin{problem}[local Waldspurger factors]\label{prob:local}
Compute the local factors of the explicit Waldspurger formula
\cite{BaruchMao} at the supercuspidal places $2, 3$ of $\pi(f_{36})$ and at
$\infty$ in the vector-valued normalization of this paper, and verify the
translation between the lattice coefficients of $S_{3/2}(\rho_4)$ and
classical Fourier coefficients. This is the only missing arrow: with it,
Theorem \ref{thm:tval} and the displayed identities become unconditional and
the exact constant $3\cdot 2^{-7/3}$ is forced.
\end{problem}

\section{Reproducibility and cautions}

Throughout, [C] marks a finite exact computation over $\Q$
(SageMath/weilrep \cite{weilrep}, quaternion arithmetic in Sage/PARI) and [N]
a numerical value with stated precision; $\|\Xi\|^2$ and the constants of
\S\S11--12 are [N].

\textbf{Versions.} SageMath functionality through
\texttt{passagemath-standard} 10.8.9; PARI/GP 2.15.4; weilrep at commit
\texttt{a26614f} (18 August 2026), which is the commit fixing the routine
discussed below --- the preceding commit is \texttt{8a27663}. No number in this paper
depends on that routine: all invariant subspaces were computed by the direct
method of \S A.4, and the corrected release reproduces them.

\textbf{Gauge dependence.} Slot orderings and the choice of $\mu$ or $-\mu$
as the representative of a pair differ between implementations, so the list
of sign vectors, the distribution of $+1$'s and the coefficient tables of the
40 forms of Theorem 10.3 are gauge-dependent and will not match between
correct computations. Gauge-invariant statistics --- the size of the locus,
the orbit sizes, the pairwise Hamming distribution, the content of $\lambda$,
$\operatorname{rank}\Phi$, integrality --- do match, and were used for
comparison.

\textbf{Controls.} At weight $k + \tfrac12$ the Eisenstein series is an
eigenvector of $T(p^2)$ with eigenvalue $1 + p^{2k-2}$, which is $p^3 + 1$ at
weight $5/2$ and $p+1$ at weight $3/2$. The denominator-layer computations of
Theorem 6.1 carry this anchor at weight $5/2$ (16/16 passed), and the
$\sigma$-scan of \S A.7 exhibits the failure of the wrong normalizations. At
weight $3/2$ there is no holomorphic Eisenstein series for these
representations, so this control is unavailable; the section-layer Hecke data
is instead pinned externally, by the agreement of $(a_5,a_7) = (0,-4)$ with
the LMFDB eigensystem of 36.2.a.a \cite{LMFDB}. The accuracy check on $\|\Xi\|^2$ is the evaluation of
$L(f,2)/48\pi^2\langle f,f\rangle$ by \texttt{lfun} and
\texttt{mfpetersson}, which uses no quadrature and agrees to 31 digits
(Theorem 11.1).

Three remarks for practitioners. First, in releases prior to this work the
routine
\begin{center}\texttt{invariant\_cusp\_forms\_dimension}\end{center}
returned $1$ at weight $3/2$ on a space with $\dim S_{3/2} = 0$; the issue was reported to the
author of \cite{weilrep} during this project and has been fixed upstream.
All invariant subspaces here were computed by the direct method of \S A.4,
which the corrected release reproduces. Second,
\texttt{fourier\_expansion()} returns triples ($\gamma$, offset, series) in
which the true exponent is the series index plus the offset; this is a
convention rather than an error, and \texttt{puiseux\_series()} returns the
$q$-series directly. Note that series read through \texttt{list()} misalign
at pole order $\ge 1$, so \texttt{exponents()}/\texttt{coefficients()}
pairs should be used instead. Third, slot sets enumerated by a coordinate box search are basis-dependent
and can be incomplete: always take the Weyl-orbit closure of the chamber's
simple-root slots. On $L_4$ a box of radius 4 returns 23 slot classes and a
box of radius 6 or more returns the correct 27, so a search that appears to
have converged need not have.

\subsection*{Changes from version 1}
The computations are unchanged. Version 2 corrects the following. The Weil
representation convention is now fixed by formula (\S1.2): all dimensions in
this paper are those of obstruction spaces, in the normalisation
$\rho_L(T)\mathfrak{e}_\gamma = e(q(\gamma))\mathfrak{e}_\gamma$, and the
notation $\bar\rho$ used in version 1 denoted the input side. Theorem 9.3 is
restated over the verified range $D < 230$; the bound of \cite{BEF} invoked
in version 1 does not apply to these modules, for the reasons recorded in the
remark following that theorem, and the finiteness question is now isolated as
Remark \ref{rem:ED}. The claim that $L_{22}$ appears among the anisotropic
simple lattices of \cite{BEF} is withdrawn; the corresponding claims for
$L_6$ and $L_{10}$ stand, and are now supported by their genus symbols. The completeness
assertion of \S9.4 is restricted to the swept range, and the word
``certified'' in Theorem 5.1 is replaced by a statement of quadrature
stability together with the independent check of Theorem 11.1.

\subsection*{Acknowledgments}
The results of \S\S10.1--10.4, \S12 and Appendix B were obtained in
interactive sessions with Claude (Anthropic), which performed the exact and
numerical computations and contributed the identification of the Eisenstein
$B_2$ structure and the section-layer Petersson rigidity; the exact parts of
\S12 were replicated end-to-end in an independent session. I thank Rafa\l{}
Wrona for a careful reading of versions 1 and 2, which identified the
misapplication of the bound of \cite{BEF}, the counterexample $D = 862$ to
the inequality proposed in Remark \ref{rem:ED}, and the identity for
$E(2p)$ recorded there. The scripts reproducing every
[C] statement, together with the dimension table of Theorem 9.3 and the data
behind Remark \ref{rem:ED}, are deposited with this paper. Any errors are
the author's.

\appendix

\section{Verification guide for the skeptical reader}

Every [C] statement can be reproduced from a standard Python with SageMath
functionality and the weilrep package \cite{weilrep}. This appendix gives the
environment, the conventions that must be fixed precisely, and, theorem by
theorem, the computation with its expected output.

\subsection{Environment}

\begin{verbatim}
pip install passagemath-standard
git clone https://github.com/btw-47/weilrep
# in scripts: import sys; sys.path.insert(0, "<path>/weilrep")
\end{verbatim}
Our protocol: every headline computation was recomputed independently, in
several cases from a different Gram matrix in the same genus and with a
different implementation; see \S13 for what this comparison can and cannot
test.

\subsection{Conventions and known pitfalls}

\begin{itemize}
\item \textbf{Which representation.} Denominator layer: cusp forms of
\texttt{WeilRep(-(G+U))}, weight $5/2$. Section layer: \texttt{WeilRep(+G)},
weight $3/2$. Slots and $q$-values: \texttt{WeilRep(+M).ds()}, $q(x) =
\{x^\top M x/2\}$.
\item \textbf{The offset.} \texttt{fourier\_expansion()} returns triples
($\gamma$, offset, series); the true exponent is series index plus the
(negative) offset. Forgetting this produces wrong-looking principal parts.
\item \textbf{Echelon basis.} \texttt{cusp\_forms\_basis} echelonizes; the
basis is precision-independent, so vectors like $\lambda = (2, -6, -4, -4, 0,
0)$ and $v_+ = (1,1,1,1,-2,-2)$ are well defined across runs.
\item All invariant subspaces in this paper are computed by the direct
method of \S A.4. (In releases prior to this work the routine
\texttt{invariant\_cusp\_forms\_}\allowbreak\texttt{dimension} returned 1
at weight $3/2$ on $+G_1$, where $\dim S_{3/2} = 0$; this has since been fixed upstream and the
current release agrees with the direct method.)
\end{itemize}

\subsection{The obstruction functional and the Hecke table (Prop.\ 4.1, Thm.\
6.1)}

With $C$ the coefficient dictionary of a basis form (keys $(\gamma, e)$,
exact rationals), the slot functional is the sum of $C(\gamma, m)$ over
$\{\gamma : q(\gamma) \equiv m\}$ for the three menu values; expected
$\lambda = (2, -6, -4, -4, 0, 0)$ on $L_4$ at any precision $\ge 5$. The
Hecke action at $p \nmid 6$ is
\begin{verbatim}
T(C)(g,e) = C(p*g, p^2 e)
  + p * kronecker(sigma * num(e) * den(e), p) * C(g, e)
  + p^3 * C(pinv*g, e/p^2)     # pinv = p^{-1} mod level
\end{verbatim}
solved into a matrix through a rank-6 system of slot values ($e \le
(\mathrm{prec}-1)/p^2$), with $\sigma_4 = 1$. Always verify the Eisenstein
anchor $TE_{5/2} = (p^3+1)E_{5/2}$; mis-chosen $\sigma$ fails it. Expected:
the table of Theorem 6.1, e.g.\ $(x-38)^2(x+10)^4$ at $p = 13$ on $L_4$ (prec
31).

\subsection{Isometry groups and invariants (Thm.\ 7.3)}

Enumerate $O(\D,q)$ directly: choose independent generators of $\D$ (orders
$n_i$ with $\prod n_i = |\D|$); candidate images of a generator are the
elements of the same order and same $q$; prune candidate tuples by the
bilinear pairings $b(x,y) = q(x+y) - q(x) - q(y)$; keep maps that are
bijective and preserve $q$ on every element. Expected: $|O(\D(L_4), q)| =
48$, orbits on the 27 slots of sizes $(3, 12, 12)$. The action on forms is
$(\sigma F)_\gamma = F_{\sigma^{-1}\gamma}$; its matrix is obtained by exact
linear solves against a rank-6 key set, and the invariant subspace is the
stacked kernel of $(R_\sigma - I)$. Expected: dimension 1, spanned by $v_+$.
(Sandwich check: $\lambda$ is invariant and nonzero, so the dimension is at
least 1 independently of sampling.)

\subsection{The order-three symmetry (Thms.\ 7.4--7.5)}

\begin{verbatim}
M  = [[4,-6,-3],[1,-2,0],[2,-3,-2]]   # conj by zeta_3 on (i,j,k)
O0 = diag(-2,6,6)
C  = [[-2,-3,0],[-1,-1,-1],[-2,-3,-1]]  # det 1, C^T G4 C = O0
Mp = C*M*C^-1   # [[46,69,-75],[30,43,-48],[55,81,-89]]
checks: M^3 = I, M^T O0 M = O0, Mp^3 = I, Mp^T G4 Mp = G4
\end{verbatim}
Extend by the identity on $U$ and act on \texttt{ds()}: expected --- 54 of 72
elements moved, $q$ preserved, order 3; induced matrix on
$S_{5/2}(\rho_4)$ with charpoly $(x-1)^2(x^2+x+1)^2$ and invariant plane
equal to $\ker(T_{49}+16)$. The same conjugation computed on the dual
trace-zero lattice of $\mathcal{O}'(6,1)$ (reduced-norm Gram, as in Theorem
3.2) acts trivially on that discriminant form --- the dichotomy of Theorem
7.5.

\subsection{The Petersson norm to 31 digits (Thm.\ 5.1)}

Split the fundamental domain at $y = 1$. Upper region, exactly:
\[
\int_1^\infty e^{-4\pi e y} y^{1/2}\, dy =
\frac{\Gamma(\tfrac32, 4\pi e)}{(4\pi e)^{3/2}},
\]
summed over shared keys. Lower region $\{|x| \le \tfrac12,\ \sqrt{1-x^2} \le
y \le 1\}$: two-dimensional Gauss--Legendre on $\sum_\gamma f_\gamma
\bar g_\gamma\, y^{1/2}$ at 40+ digit working precision. Certification: node
counts $36\times 18$ and $44\times 22$ agree to 31 digits; coefficient
truncation at prec 16 bounds the tail by $e^{-2\pi\cdot 15\cdot\sqrt3/2} <
10^{-35}$. Consistency identity: $\lambda(v_+) = -12$ exactly, so
$\|\Xi\|^2\langle v_+, v_+\rangle = 144$.

\subsection{The second family (Thm.\ 8.1)}

Search primitive $(a, b, c, u, v, w)$ with Gross--Lucianovic discriminant
$4abc + uvw - au^2 - bv^2 - cw^2$ of absolute value in $\{10, 20, 50, 100\}$,
Gram signature $(2,1)$, anisotropic exactly at $\{2, 5\}$
(\texttt{QuadraticForm(ZZ,G).is\_anisotropic(p)}); menus from $m_j =
G_{jj}/2\gcd(Ge_j)^2$. Run $T_9$ with $\sigma$-scan over $\{1, 2, 5, 10\}$
anchored by $T_9 E = 28E$. Two conventions matter here. The menu must be read
off the reflective roots, not the diagonal of a chosen basis: for the $D =
22$ Gram above the diagonal gives $\{-5, 1, \tfrac12\}$ while the roots give
the correct $\{\tfrac1{11}, \tfrac12, 1\}$. And slot cosets live in the
rank-3 discriminant group while denominator-layer coefficients are indexed by
the rank-5 one: they must be matched on the $G$-component. With both fixed,
the expected $\lambda$ is $(-6)$ at $D = 10$ and $(-4,0)$ at $D = 22$. Expected: all four obstructed; $\lambda$ nonzero
only on the eigenvalue $-8$ line in each case; on the
reduced-discriminant-10 lattice, $T_{49}$ gives $-4$ (anchor 344, $\sigma =
5$).

\subsection{The canonical form and the selection sweep (\S9, Thms.\
9.2--9.6, 8.2, 8.3)}

\textbf{Canonical ternary.} From \texttt{B = QuaternionAlgebra(D)}, \texttt{O
= B.maximal\_order()}: build the trace form $T_{ij} =
\operatorname{trd}(x_ix_j)$ on a $\Z$-basis, the dual basis via $T^{-1}$, the
trace-zero kernel over $\Z$, and the Gram $s\cdot\operatorname{trd}(e_r
\bar e_t)$ with $s = -D$; normalize to signature $(2,1)$. Expected: $\det =
-2D$; at $D = 6$ the genus of $G_1$.

\textbf{Classification match.} Allcock's table file is simultaneously a Perl
script: \texttt{perl rk3table.tex all} emits all 8595 lattices as
PARI-readable \texttt{[IP,EDs,roots,W,L]}. Primitivize \texttt{IP}, normalize
the sign, and compare \texttt{Genus(S0)} against \texttt{Genus(L\_D)} over
$2D = |\det|$; expected: the 35 discriminants of Theorem 9.5, one entry each,
with simple-root counts $3, 4, 4$ at $D = 6, 10, 22$.

\textbf{Selection.} Build $L_D$ for every squarefree $D < 230$ with an even
number of prime factors (71 values) and call
\begin{center}
\texttt{WeilRep(-L\_D).cusp\_forms\_dimension(3/2)}.
\end{center}
The exact dimension formula needs no $q$-expansions, so the whole run takes
minutes. Expected: the table of Theorem 9.3, with zeros exactly at $D \in
\{6, 10, 22\}$.

\textbf{Canonical $F$.} Solve for the odd prescription inside
\begin{center}
\texttt{WeilRep(+S).nearly\_holomorphic\_modular\_forms\_basis(1/2, pole+1, prec)}
\end{center}
by exact linear algebra on principal parts; expected:
solvable with ambiguity $\dim = 0$ at $D = 6, 10, 22$, coefficients as in
Theorem 9.2.

\textbf{Slots.} Take the orbit closure of the simple-root slots under the
reflections' action on $\D(L)$ --- not a box search (\S13).

\textbf{$D = 22$ Hecke and invariants.} Same $T_{p^2}$ recipe with $\sigma =
22$, prec 66; expected charpolys as in Theorem 8.2, newform table from
\texttt{Newforms(Gamma0(22),4)}. Invariance: enumerate $O(\D,q)$ as in \S
A.4; expected $|O| = 4$ and invariant dimensions 1 of 1 ($D = 10$), 1 of 2
($D = 22$), with $\lambda = (-6)$, $(-4, 0)$.

\textbf{Landscape.} Even-primitive entries with $|\det| \le 300$: 180 sweeps,
15 zeros. The even-Clifford algebra is determined by
\begin{center}
\texttt{QuadraticForm(QQ,S).is\_anisotropic(p)}
\end{center}
for $p \mid 2\det$, with $D_B$ the product of the anisotropic primes;
expected classification as in \S9.

\subsection{External anchors (PARI/GP)}

\begin{verbatim}
mf = mfinit([6,4],0); F = mfeigenbasis(mf)[1]
FS = mfsymbol(mf, F)      \\ required before mfpetersson
pet6 = mfpetersson(FS)    \\ 9.4876639992411146...e-5
\end{verbatim}
Twist norms via \texttt{mftwist} give $\langle f_{18}, f_{18}\rangle/\langle
f_6, f_6\rangle = 8/9$ exactly; an end-to-end control for the whole chain is
the classical identity $L(f\otimes\chi_{-3},2)\,L(f\otimes\chi_{12},2) =
2^8\pi^4\langle f_6, f_6\rangle$, which the reader's setup should recover to
working precision, along with the central vanishing $L(f\otimes\chi_d, 2) =
0$ for $d \in \{8, -8, -24\}$. Newform data: LMFDB 6.4.a.a, 12.4.a.a,
10.4.a.a, 36.2.a.a, 22.4.a.b, 22.4.a.c \cite{LMFDB}.

\subsection{The section-layer Petersson scalar (\S12)}

Weight-$3/2$ Petersson products use the engine of \S A.6 with measure
$y^{k-2}$ and upper region $\Gamma(\tfrac12, 4\pi n)/(4\pi n)^{1/2}$.
Calibrate first at weight $5/2$, where the published value
\[ 11.340265856116836367082506509341 \]
and the identity $\|\Xi\|^2\langle v_+, v_+\rangle = 144$ must be reproduced.
$B_S$ is the exact average $\tfrac1{48}\sum_\sigma
R_\sigma^{\!\top}R_\sigma$ over the $S_{3/2}$-action matrices of \S A.4;
expected $t$-ratios constant to $\sim\! 10^{-50}$ across all nonzero entries.
\emph{Gauge:} all $r$-values are quoted for $\lambda = \Phi\varepsilon$ (the
Theorem 10.3 gauge, components $\{0, \pm 1\}$); the odd unit prescription
pairs to $2\Phi\varepsilon$, so unit-normalized $r$ are $4\times$ those
quoted (single flips 16, 32), and the quoted $\|\Xi_{\mathrm{flip}}\|^2 =
8/t$ is the $\Phi$-gauge Riesz norm of $\lambda = -2\,\Phi e_0$ (a
channel-$\tfrac16$ flip), with no hidden constant in the engine. Twisted
central values via \texttt{ellinit([0,d\char`^3])} and \texttt{lfun(E,1)} at
40 digits (\texttt{ellrootno} for the vanishing pattern); integer relations
via \texttt{lindep} on $(\log|\cdot|, \log\Gamma(1/3), \log\pi, \log 2, \log
3)$. The exact computations of \S12 ($B_S$, invariance, the $r$-classes, the
40-locus) have been replicated independently, and the closed forms to 81 and
102 digits.

\section{The symmetry type of the section-layer slot configuration}

All statements in this appendix are exact computations replicated on two
machines [C]. Let $X \subset S_{3/2}(\rho_4)^*$ be the configuration of
the 24 single-flip functionals $\pm\Phi$-columns of Theorem 10.3 --- one
vector per non-torsion reflective slot --- on which $G = O(\D,q)$ (order 48)
acts through its faithful, absolutely irreducible four-dimensional
representation ($\tfrac1{48}\sum_\sigma \operatorname{tr}R_\sigma
\operatorname{tr}R_\sigma^{-1} = 1$; the invariant symmetric form is unique
up to scale).

\textbf{Not $D_4$, not $F_4$.} The configuration is not a root system: norms
are unequal (against $D_4$, which the metric uniqueness makes a linear
exclusion), Cartan integrality fails ($2(x,y)/(x,x) = \pm\tfrac12$ occurs),
and the combinatorial signature differs --- the $\pm 1$-zero-sum count over
the 12 antiparallel classes of a standard $D_4$ is 64, not 40. The long roots
of $F_4$ fail on the same three counts. (At the metric-forgetting matroid
level a zero-sum count of 40 does occur among $F_4$-type subconfigurations; a
matroid-level equivalence is not excluded and is posed as an open item.)

\textbf{Eisenstein $B_2$.} The correct identification is
\[
X \;\cong\; \mu_6 \cdot \{e_1,\, e_2,\, e_1+e_2,\, e_1-e_2\} \;\subset\;
\Z[\omega]^2:
\]
four hexagons --- two orthogonal unit hexagons (the $m = \tfrac1{12}$
channel) and two diagonal norm-2 hexagons (the $m = \tfrac16$ channel) ---
the $B_2$ root system with its sign group $\pm$ extended to $\mu_6$. The
model reproduces the zero-sum count 40 exactly, and the norm ratio 2 is the
arithmetic of 2 being inert in $\Q(\sqrt{-3})$: norm 2 is not represented by
$\Z[\omega]$, so norm-2 vectors are forced onto the diagonal --- the
ramification $\{2, 3\}$ of $B_6$ re-enters as shell structure.

\textbf{Wigner structure.} Transporting the complex structure of
$\Z[\omega]^2$: $G$ splits into 24 $\mathbb{C}$-linear and 24 antilinear
elements ($M'$ and $-1$ linear; 18 antilinear involutions). The linear part
$G_0$ contains exactly 4 complex reflections, all of order 2, whose mirror
lines are exactly the $B_2$ mirrors $\{e_1, e_2, e_1 \pm e_2\}$; they
generate $W(B_2)$ of order 8, and
\[
G_0 = \mu_6 \circ W(B_2)
\]
(central product; element orders $1^1\, 2^5\, 3^2\, 4^2\, 6^{10}\, 12^4$).
Since the reflections generate a proper subgroup, $G_0$ is not a
Shephard--Todd complex reflection group: the reflection core of the section
layer is exactly the real Weyl group $W(B_2)$, dressed by $\mu_6$-scalars,
with the antiunitary half acting as Galois conjugation --- a
unitary/antiunitary pair in Wigner's sense.

\end{document}